\documentclass[12pt,twoside,reqno]{amsart}
\usepackage[all]{xy}
        \usepackage {amssymb,latexsym,amsthm,amsmath,mathtools,dsfont,mathrsfs, bm}
        \usepackage{enumitem,color}

\usepackage{makecell}
\usepackage{mathtools}

\usepackage[backend=biber,maxbibnames=999,doi=false,isbn=false,url=false]{biblatex} 
\usepackage{float}
\usepackage{hyperref}
\usepackage{framed}
\long\def\symbolfootnote[#1]#2{\begingroup%
\def\thefootnote{\fnsymbol{footnote}}\footnote[#1]{#2}\endgroup}

\makeatletter
\def\imod#1{\allowbreak\mkern10mu({\operator@font mod}\,\,#1)}
\makeatother
\makeatletter
\renewcommand*\env@matrix[1][*\c@MaxMatrixCols c]{%
  \hskip -\arraycolsep
  \let\@ifnextchar\new@ifnextchar
  \array{#1}}
\makeatother
\usepackage[capitalize,nameinlink]{cleveref} 
\usepackage{comment} 
\usepackage{xcolor} 
\hypersetup{
    colorlinks,
    linkcolor={red!80!black},
    citecolor={green!80!black},
    urlcolor={blue!80!black}
}

\newtheorem{theorem}{Theorem}[section]
\newtheorem{lemma}[theorem]{Lemma}

\newtheorem{proposition}[theorem]{Proposition}
\newtheorem*{theorem*}{Theorem}
\theoremstyle{definition}
\newtheorem{defn}[theorem]{Definition}
\newtheorem{definition}[theorem]{Definition}
\newtheorem{remark}[theorem]{Remark}
\newtheorem{example}[theorem]{Example}

\numberwithin{equation}{section}
\newcommand{\ignore}[1]{}

\newcommand{\mynote}[1]{}
\newcommand{\bmf}[1]{\mathbf{#1}}
\newcommand{\lbd}{\lambda}
\newcommand{\Lbd}{\bm {\lambda}}
\newcommand{\Ceil}[1]{\left\lceil #1 \right\rceil}
\title[ $z$-Classes in Finite Coxeter Groups]{$z$-classes in Finite Coxeter Groups}
\author{Dilpreet Kaur}
\email{dilpreetkaur@iitj.ac.in}
\address{Indian Institute of Technology Jodhpur
N.H. 62, Nagaur Road, Karwar Jodhpur 342030
Rajasthan}
\author{Uday Bhaskar Sharma}
\email{udaybsharmaster@gmail.com}
\address{School of Advanced Engineering, UPES, Dehradun, India}

\thanks{We are thankful to Anupam Singh for his useful comments on this work. We are also thankful to Pooja Singla for her lecture series on Coxeter groups in conference 
"Recent Developments in Group Theory" at IIT Bhubaneshwar in July 2025. The first named author would like to acknowledge the support of SERB through MATRICS project MTR/2022/000231.}
\date{\today}
\subjclass[2020]{}
\keywords{}
\begin{document}
\setcounter{section}{0}
\begin{abstract}
We give the enumeration of $z$-classes in finite Coxeter groups. 
\end{abstract}
\maketitle
\section{Introduction}
Let $G$ be a group. Two elements $g_1,g_2\in G$ are called $z$-equivalent or $z$-conjugate if their centralizers $Z_G(g_1)$ and $Z_G(g_2)$ are conjugate in $G.$ Clearly, 
$z$-equivalency is an equivalence relation on $G$, equivalence classes are called $z$-classes or centralizer classes. In general, $z$-equivalency is weaker relation than conjugacy relation on group $G.$ The notion of $z$-classes was introduced by Ravi Kulkarni in \cite{RK:2007}. Since then, it has been studied for various families of groups. For more information on $z$-classes, one may refer to the survey article \cite{BS}.

In \cite{BKS:2019}, the authors studied the $z$-classes of symmetric groups and Alternating groups. They also give the count of $z$-classes for these groups. Since $S_{n+1}$ is a Coxeter group of type $A_n,$ it is natural to ask the number and structure of $z$-classes in other finite Coxeter groups. For more information on Coxeter groups, reader may refer to  \cite{Hum}.

Every finite Coxeter group is the direct product of finitely many of irreducible finite Coxeter groups. In Theorem \ref{Direct_Product}, we show that the $z$-classes of direct product of two groups $G_1$ and $G_2$ is the cartesian product of $z$-classes of $G_1$ and $z$-classes of $G_2.$ Clearly, this result can be extended to direct product of finitely many groups. Hence, it is enough to compute the $z$-classes in irreducible finite Coxeter groups.

 We know that the finite Coxeter groups of type $B_n$ and $C_n, $ are isomorphic, and they both are isomorphic to wreath product $C_2 \wr S_n$ of cyclic group $C_2$ of order $2$ and symmetric group $S_n$ of degree $n.$ The following theorem gives the number of $z$-classes in $C_2 \wr S_n.$
\begin{theorem}\label{ThMain1}
The total number of $z$-classes in $C_2 \wr S_n$ is:
\begin{equation}\label{Emain1}\sum_{{\Lbd} \vdash n}\left( \left(\prod_{i=1}^u \left(\left\lfloor \frac{l_i}{2}\right\rfloor + 1\right)\right)\left(\prod_{j=1}^v (w_j+1)\right)\right),\end{equation}
where ${\Lbd} = k_1^{l_1}\cdots k_u^{l_u}m_1^{w_1}\cdots m_v^{w_v}$ is a partition of $n$, with $1 \leq k_1 < k_2 < \cdots < k_u \leq n$ being odd positive integers and $2\leq m_1 < m_2 < \cdots <m_v \leq n$ being even positive integers.
\end{theorem}
The above theorem is proved in the section \ref{$z$-Classes of $C_n$}.

The finite Coxeter group of type $D_n$ is a normal subgroup of $C_2\wr S_n$ of index $2.$  To describe the number of $z$-classes Coxeter group of type  in $D_n, $ we need some type of restricted partitions of $n.$ For any integer $n,$ let $\zeta(n)$ denotes the number of partitions of $n$ with all even parts greater then or equal to $4.$ For example $\zeta(8)=2$ as $4^2$ and $8$ are the only restricted partitions of $8$ of required type. For any integer $n,$ $\delta(n)$ denotes the number of partitions of $n,$ with some ( at least one) odd part, with odd multiplicity. It is easy to see that if $n$ is an odd number, then $\delta(n)$ is same as total number of partitions of $n.$
However, $\delta(2)=0,$ and $\delta(4)=1.$ Next, let $\delta^\prime(n)$ denotes the number of partitions of $n,$ with either all even parts or its all odd parts have even  multiplicity. It is easy to see that if $n$ is an odd number, then $\delta^\prime(n)=0,$ on the other hand $\delta(2)=2,$ and $\delta(4)=4.$ It is clear that $\delta(n)+\delta^\prime(n)$ is same as the total number of partitions of $n.$

The following theorem gives the number of $z$-classes in Coxeter group of type $D_n.$
\begin{theorem}\label{DMain2}
Let ${\Lbd} = k_1^{l_1}\cdots k_u^{l_u}m_1^{r_1}\cdots m_v^{r_v}$ be a partition of $n$, with $1 \leq k_1 < k_2 < \cdots < k_u \leq n$ being odd positive integers and $2\leq m_1 < m_2 < \cdots <m_v \leq n$ being even positive integers.
The total number of $z$-classes in $D_n$ is:
\begin{enumerate}
    \item If $n$ is odd, then
\begin{equation}\label{Emain1}\sum_{{\Lbd} \vdash n}\left( \left(\prod_{i=1}^u \left(\left\lfloor \frac{l_i}{2}\right\rfloor + 1\right)\right)\left(\prod_{j=1}^v (r_j+1)\right)\right),\end{equation}
\item If $n$ is even, then\begin{equation}\label{Emain2}\begin{matrix}\sum_{\Lbd \in \delta(n) }\left( \left(\prod_{i=1}^u \left(\left\lfloor \frac{l_i}{2}\right\rfloor + 1\right)\right)\left(\prod_{j=1}^v (r_j+1)\right)\right)+\\ \sum_{{\Lbd} \in \delta^\prime(n) }\left\lceil \frac{\left(\prod_{i=1}^u \left(\left\lfloor \frac{l_i}{2}\right\rfloor + 1\right)\right)\left(\prod_{j=1}^v (r_j+1)\right)}{2}\right\rceil-\zeta(n-2)+|\delta^\prime\left(\frac{n}{2}\right)|\end{matrix},\end{equation}
where $\delta(n)$ denotes the set of partitions of $n$, where some odd part has odd multiplicity. $\delta'(n)$ denotes the set of partitions of $n$, where all the odd parts have even multiplicity. $\zeta(n-2)$ denotes the number of partitions of $n-2$, with only even parts, but does not have 2 as a part. 
\end{enumerate}\end{theorem}

The above theorem is proved in the section \ref{$z$-Classes of $D_n$}.

The Coxeter groups of type $I_2(n)$ are dihedral groups of order $2n.$ The $z$-classes for dihedral groups have been computed in Example 5.1 in \cite{BS}. We compute the same in theorem \ref{z-classes-Dihedral-group} for the sake of completeness.

The number of $z$-classes in finite Coxeter groups of exceptional types are given in the table \ref{Table-exceptional}. We used the computer algebra software GAP \cite{GAP4} and its package SLA \cite{SLA} for computing this table. The code used to compute $z$-classes is given at the end of section \ref{z-classes_of_exceptional_types}. 

Since every Weyl group can be realized as a Coxeter group. This article also solves the Problem 5.6 given in the survey article \cite{BS}.

\section{Preliminaries}

In this section, we collect some basic results regarding $z$-classes of a groups, which will be used in later sections. Since the finite Coxeter groups are a direct product of irreducible finite Coxeter groups, the following theorem implies that it is enough to compute the number of $z$-classes in irreducible finite Coxeter groups.

\begin{theorem}\label{Direct_Product}
     Let $G_1$ and $G_2$ be two finite groups. Let $z_1$ be the number of $z$-classes of $G_1$ and $z_2$ be the number of $z$-classes of $G_2$. Then the number of $z$-classes of the direct product $G_1 \times G_2$ is $z_1z_2$.
 \end{theorem}
 \begin{proof} 

     By the definition of $z$-classes, two elements $(a_1,a_2)$ and $(b_1, b_2)$ in $G_1 \times G_2$ are $z$-conjugate iff $Z_{G_1\times G_2} ((a_1,a_2))$ is conjugate to $Z_{G_1 \times G_2}(b_1, b_2).$ 
     This means that $\exists$ $(g_1, g_2) \in G_1 \times G_2$ such that $$(g_1,g_2)Z_{G_1\times G_2} ((a_1,a_2))(g_1^{-1},g_2^{-1}) = Z_{G_1 \times G_2}((b_1, b_2))$$

     But we know that $Z_{G_1 \times G_2}((a_1,a_2)) = Z_{G_1}(a_1) \times Z_{G_2}(a_2)$. Likewise for $Z_{G_1 \times G_2}((b_1, b_2))$.
     Hence we have 
     \begin{eqnarray*}
        Z_{G_1}(b_1) \times Z_{G_2}(b_2) &=& Z_{G_1 \times G_2}((b_1, b_2))\\
        &=& (g_1,g_2)Z_{G_1\times G_2} ((a_1,a_2))(g_1^{-1},g_2^{-1})\\
        &=& (g_1Z_{G_1}(a_1)g_1^{-1}) \times (g_2Z_{G_1}(a_2)g_2^{-1})
     \end{eqnarray*}
     This means that $Z_{G_1}(b_1) = g_1Z_{G_1}(a_1)g_1^{-1}$  and $Z_{G_2}(b_2) = g_2Z_{G_2}(a_2)g_2^{-1}$. Thus $a_1$ and $b_1$ are $z$-conjugate in $G_1$, and $a_2$ and $b_2$ are $z$-conjugate in $G_2$. 

     If $a_1$ and $b_1$ are $z$-conjugate in $G_1$, and $a_2$ and $b_2$ are $z$-conjugate in $G_2$, then it is clear that $(a_1, a_2)$ and $(b_1, b_2)$ are $z$-conjugate in $G_1 \times G_2$. Thus, $$\{\tau: \tau\text{ is a $z$ class in }G_1\times G_2 \} = \{\tau_1: \tau_1\text{ is a $z$ class in }G_1\} \times \{\tau_2: \tau_2\text{ is a $z$ class in }G_2\}  $$

     Thus the number of $z$ classes in $G_1$ and $G_2$ is $z_1z_2$.
 \end{proof}

As the Coxeter group $D_n$ is an index 2 subgroup of Coxeter group $C_2 \wr S_n$ of type $C_n.$ We recall some connections between $z$-classes of a group $G$ with $z$-classes of a its subgroup $H$ with $[G:H]=2.$ We know that a conjugacy class $C$ in $G$ is either a conjugacy class in $H$ as well, or $C$ decomposes into two conjugacy classes in $H.$ If conjugacy class $C$ in $H$ is also a conjugacy class in $G,$ then we call that $C$ is a non-split conjugacy class of $H,$ otherwise it is called split conjugacy class.

\begin{lemma}\label{Lemma_split}
Let $H$ be a subgroup of $G$ such that $[G:H]=2.$ Let $x,y \in H$ and both lie in split conjugacy classes. If $x$ is $z$-conjugate to $y$ in $H,$ then they are $z$-conjugate in $G.$
\end{lemma}
\begin{proof}
    As $x,y\in H$ both lie in split conjugacy classes, we have $Z_H(x)=Z_G(x)$ and $Z_H(y)=Z_G(y).$ Now if $x$ is $z$-conjugate to $y$ in $H,$ then there exist $h\in H,$ such that $Z_H(x)=hZ_H(y)h^{-1},$ which implies $Z_G(x)=hZ_G(y)h^{-1}.$
\end{proof}
\begin{lemma}\label{Lemma1_index2}
Let $H$ be a subgroup of $G$ such that $[G:H]=2,$ and $C_1,C_2$ be two non-split conjugacy classes of $H.$ Let $x\in C_1$ and $y\in C_2.$ If $x$ and $y$ are $z$-conjugate in $G, $ then $x$ and $y$ are $z$-conjugate in $H.$
\end{lemma}
\begin{proof}
    As $C_1$ and $C_2$ are non split conjugacy classes in $H,$ we have $Z_H(x)=Z_G(x)\cap H, $ and $Z_H(y)=Z_G(y)\cap H, $ where $x\in C_1$ and $y\in C_2.$ Let $x$ and $y$ be $z$-conjugate in $G, $ there exist $g\in G,$ such that $Z_G(x)=gZ_G(y)g^{-1}.$ If $g\in H,$ then we are done, else choose $g^\prime\in Z_G(y)\setminus H,$ and take $h=gg^\prime.$ Now $$Z_G(x)=gZ_G(y)g^{-1}=gg^\prime Z_G(y)(g^\prime)^{-1}g^{-1}=hZ_G(y)h^{-1}=Z_G(hyh^{-1})$$ 
    Now, taking intersection with $H$ on both sides, we get $Z_H(x)=Z_H(hyh^{-1})=hZ_H(y)h^{-1}.$ This completes the proof.
\end{proof}
With same notations as Lemma \ref{Lemma1_index2},
we have if $Z_G(x)=Z_G(y),$ 
then $Z_G(x)\cap H = Z_G(y)\cap H,$ and hence $Z_H(x)=Z_H(y).$ 
The converse does not hold, for example consider Alternating group $A_6$ of degree 6, we know that $A_6$ is a subgroup of $S_6,$ and $[S_6; A_6]=2.$ It is easy to see that $Z_{A_6}((4,5,6))=Z_{A_6}((1,2,3)(4,5,6)),$ but $Z_{S_6}((4,5,6))\neq Z_{S_6}((1,2,3)(4,5,6)).$ The following lemmas give the partial converse.

\begin{lemma}\label{Lemma1_index2_converse}
With the same notation as Lemma \ref{Lemma1_index2}, If $Z_H(x)=Z_H(y),$ then $Z_G(x)=Z_G(y)$ if and only if $(Z_G(x)\cap Z_G(y))\setminus H\neq \Phi.$
\end{lemma}
\begin{proof}
Let $Z_G(x)=Z_G(y),$ then $(Z_G(x)\cap Z_G(y))=Z_G(x),$ as $x\in C_1,$ and $C_1$ is a non-split conjugacy class of $H.$ We have $Z_G(x)\cap H \neq \Phi.$
Conversely let $Z_H(x)=Z_H(y),$ and $(Z_G(x)\cap Z_G(y))\setminus H\neq \Phi.$ Let $g \in (Z_G(x)\cap Z_G(y))\setminus H.$ Then $Z_G(x)=Z_H(x)\cup gZ_H(x)$ and $Z_G(y)=Z_H(y)\cup gZ_H(y).$ Using $Z_H(x)=Z_H(y),$ we get $gZ_H(x)=gZ_H(y),$ and hence $Z_G(x)=Z_G(y).$
\end{proof}
\begin{lemma}\label{Lemma2_index2_converse}
With the same notation as Lemma \ref{Lemma1_index2}, If $x$ and $y$ are $z$-conjugate in $H,$ and $Z(Z_G(x))\cap H=Z(Z_H(x)),$ then $x$ and $y$ are $z$-conjugate in $G.$
\end{lemma}
\begin{proof}
    Let $g\in H$ such that $Z_H(x)=gZ_H(y)g^{-1}.$
    As $x$ and $y,$ both belong to non-split conjugacy classes, we have $|Z_G(x)|=2|Z_H(x)|$ and $|Z_G(y)|=2|Z_H(y)|.$ Using $Z_H(x)=gZ_H(y)g^{-1},$ we get $|Z_G(x)|=|gZ_G(y)g^{-1}|,$ and $Z(Z_H(x))=Z(gZ_H(y)g^{-1})=gZ(Z_H(y))g^{-1}.$ This implies $gyg^{-1}\in gZ(Z_H(y))g^{-1}=Z(Z_H(x)). $
    Using $Z(Z_G(x))\cap H=Z(Z_H(x)),$ we get $Z(Z_H(x))\subseteq Z(Z_G(x)),$ which implies $gyg^{-1}\in Z(Z_G(x)),$ which in turn gives us $Z_G(x)\subseteq gZ_G(y)g^{-1}.$ Hence $Z_G(x)=gZ_G(y)g^{-1}$ using $|Z_G(x)|=|gZ_G(y)g^{-1}|.$ 
    \end{proof}
The following proposition will be useful in determining the relation between center of centralizers in groups $C_2 \wr S_n$ and $D_n.$ 

\begin{proposition}\label{Regarding_center_of_centralizers}
 Let $G = G_1 \times G_2 \times \dots \times G_n$ be a finite group and $H$ be a subgroup of $G$ and $[G:H]=2.$ Let $h = (h_1, h_2, \dots , h_n)\in H$ and $h_i\in G_i$ for all $1\leq i\leq n.$ Let $\bmf{G_i}$ denote the subgroup of $G$ isomorphic to $G_i$ and $\bmf{h_i} = (1, 1, \dots, h_i, \dots 1, 1) \in H$ with $h_i$ at the $i^{th}$ position. Let $|Z_H(\bmf{h_i})\cap \bmf{G_i} | < |Z_{G_i}(h_i)|$ for at least two distinct $i, j.$ Then $Z(Z_H(h))=Z(Z_G(h))\cap H.$  
\end{proposition}

\begin{proof}
  Let $g= (g_1, g_2, \dots , g_n)\in G. $ We can write $g = \pi_{i=1}^{n}\bmf{g_i}, $ where $\bmf{g_i} = (1,\dots, g_i, \dots 1)$ with $g_i$ at the $i^{th}$ position. Since $H$ is an index $2$ subgroup of $G,$ we have $G= H \cup Hx$ for some $x\in G.$ This implies $g = \pi_{i=1}^{n}\bmf{g_i} \in H $ if and only if $\bmf{g_i}\in Hx$ for even number of $i.$ 

  Let $h = (h_1, h_2, \dots , h_n)\in H$ and $h_i\in G_i$ for all $1\leq i\leq n.$ 
  It is easy to see that $Z_H(h)=Z_G(h)\cap H.$ This implies $Z(Z_H(h))\supseteq Z(Z_G(h))\cap H.$ To prove the converse,  let $k\in Z(Z_H(h)).$ Since $Z(Z_H(h))\subseteq H\subseteq G,$ let $k = (k_1, k_2, \dots , k_n)\in H$ and $k_i\in G_i$ for all $1\leq i\leq n.$  We are done if we show that $k_i \in Z(Z_{G_i}(h_i))$ for all $1\leq i\leq n$ as $Z_G(h)= \pi_{i=1}^{n} Z_{G_i}(h_i).$
  
  We prove this by contradiction. Let $k_i \notin Z(Z_{G_i}(h_i))$ for some $i.$ This implies that there exist $g_i \in Z_{G_i}(h_i)$ such that $k_ig_i\neq g_ik_i.$ We divide the proof into two cases.
  
 \begin{enumerate}
  \item Let $\bmf{g_i}\in H.$ This implies $\bmf{g_i} \in Z_H(h)$ as $g_i \in Z_{G_i}(h_i).$ Therefore $k\bmf{g_i}=\bmf{g_i}k$ as $k\in Z(Z_H(h)).$ Using this, we conclude that $k_ig_i= g_ik_i,$ which contradicts our assumption.
  \item Let $\bmf{g_i}\notin H.$ It is easy to see that if for some $\bmf{g_k} \in Z_H(\bmf{h_j})\cap \bmf{G_j}, $ then $g_k \in Z_{G_j}(h_j).$ Given there exists $j\neq i$ such that $|Z_H(\bmf{h_j})\cap \bmf{G_j} | < |Z_{G_j}(h_j)|.$ Let $g_j \in Z_{G_j}(h_j)$ and $\bmf{g_j}\notin  Z_H(\bmf{h_j})\cap \bmf{G_j}.$ This implies $\bmf{g_j} \notin H.$ as $H$ is index $2$ subgroup of $G,$ the product $\bmf{g_i}\bmf{g_j}\in H$ and hence $\bmf{g_i}\bmf{g_j}\in Z_H(h).$ Therefore $k\bmf{g_i}\bmf{g_j}=\bmf{g_i}\bmf{g_j}k,$ and in particular we get $k_ig_i= g_ik_i,$ which contradicts our assumption.
 \end{enumerate}
\end{proof}

The Coxeter group $I_2(n)$ is isomorphic to the dihedral group of order $2n.$ The following theorem gives $z$-classes in dihedral groups.

\begin{theorem}\label{z-classes-Dihedral-group}
    A finite dihedral group $D_{2n}$ of size $2n$ for $n\geq 3$ has 3 $z$-classes if $n$ is odd, or if $n/2$ is odd; and 4 $z$ classes if $n = 4k$ for some positive integer $k$.
\end{theorem}
\begin{proof}
Let $D_{2n} = \langle a, b: a^n = b^2 = e, bab = a^{-1}\rangle$.
    
If $n$ is odd, the conjugacy classes of $D_{2n}$ are $\{e\} = Z(Z_{D_{2n}})$, $\{a, a^{n-1}\}$, $\{a^2, a^{n-2}\}$,$\cdots$, $\{a^{\lfloor n/2\rfloor}, a^{\lceil n/2 \rceil}\}$,  and $\{b,ba,ba^2,\ldots,ba^{n-1}\}$. All the conjugacy classes $\{a^j, a^{n-j}\}$ have $\langle a \rangle$ as their centralizer. Thus $a,a^2,\ldots, a^{n-1}$ are all $z$-conjugate. And, all the $ba^s$, where $0 \leq s \leq n-1$, are conjugate, hence its centralizer is of size two, and conjugate to $\{e, b\}$. Thus, this conjugacy class is a $z$-class of its own. Hence we have 3 $z$-classes.

When $n$ is even: $Z(Z_{D_{2n}}) = \{e, a^{n/2}\}$. The conjugacy classes of $D_{2n}$ are:
\begin{itemize}
    \item $\{e\}, \{a^{n/2}\}$, which form the centre of $D_{2n}$
    \item  $\{a^j, a^{n-j}\}$, where $1 \leq j \leq n/2-1$. For each of these, the centralizer is $\langle a \rangle$
    \item $\{ba^{2s}: 0\leq s \leq n/2 -1\}$. The centralizer of $b$ is $\{e, b, a^{n/2}, ba^{n/2}\}$
    \item  $\{ba^{2s +1}: 0\leq s \leq n/2 -1\}$. The centralizer of $ba$ is $\{e, ba, a^{n/2}, ba{1 + n/2}\}$
\end{itemize}

When $n/2$ is odd, i.e, $n = 4k + 2$ for some $k$. The two central elements form a $z$-class. The centralizers of the classes $\{a^j, a^{n-j}\}$ for $1 \leq j \leq n/2 - 1$ are all $\langle a \rangle$. Now, $Z_{D_{2n}}(b) = \{e, b, a^{n/2}, ba^{n/2}\}$. As $n/2$ is odd, and $ba^{n/2}$ is not conjugate to $b$. However, $Z_{D_{2n}}(ba^{n/2}) = \{e, ba^{n/2}, a^{n/2}, b\} = Z_{D_{2n}}(b)$. Thus $b$ and $ba^{n/2}$ are $z$-conjugate. Hence there are 3 $z$-classes.

When $n = 4k$, we have $n/2 = 2k$. In this case, $Z_{D_{2n}}(b)= \{e, b, a^{n/2}, ba^{n/2}\}$. As $n/2 = 2k$, $ba^{n/2}$ is conjugate to $b$. Likewise $Z_{D_{2n}}(ba) = \{e, ba, a^{n/2}, ba^{n/2 + 1}\}$. As $ba$ and $ba^{n/2 + 1}$ are conjugate to each other, the centralizers of $b$ and $ba$ are not conjugate to each other. Thus, they form 2 separate $z$-classes. Hence there are 4 $z$-classes in $D_{2n}$, when $n = 4k$.
\end{proof}

We conclude this section with the following combinatorial lemma, which is required to prove theorem \ref{DMain2}.
\begin{lemma}\label{LemmaCount}
    Let $d_1,d_2,\ldots, d_p$ be some positive integers. For each $1 \leq i \leq p$, consider integers $0 \leq t_i \leq d_i - 1.$ The number of $(t_1, t_2, \ldots, t_p)$ such that $\sum_{i=1}^p t_i $ is even is $$\left\lceil\frac{\prod_{i=1}^p d_i}{2}\right\rceil.$$
    
\end{lemma}
\begin{proof}
    We prove this lemma by induction on $p$. When $p = 1$: Suppose $d_1$ is even, then $d_1 - 1$ is odd, and there are exactly $\displaystyle\frac{d_1}{2} = \left\lceil\frac{d_1}{2}\right\rceil$ even numbers between $0$ and $d_1 - 1$. When $m_1 $ is odd, then $d_1 -1$ is even. And there are $\displaystyle\frac{d_1 + 1}{2} = \left\lceil\frac{d_1}{2}\right\rceil$ even numbers between $0$ and $d_1 - 1$.

    Assume induction upto $p$. Given $d_1, d_2, \ldots, d_{p+1}$. Suppose $d_1d_2\cdots d_{p+1}$ is even. We may assume that $d_{p+1}$ is even. Let $(t_1,\ldots,t_{p+1})$ be a $p+1$-tuple such that their sum is even. If $t_{p+1}$ is even, then $\sum_{i=1}^p t_i$ is even. By induction, we know that there are $\left\lceil\frac{\prod_{i=1}^p d_i}{2}\right\rceil$ such $p$-tuples. As there are $\displaystyle\frac{d_{p+1}}{2}$ choices for $t_{p+1},$ we have $\displaystyle\left\lceil\frac{\prod_{i=1}^p d_i}{2}\right\rceil\frac{d_{p+1}}{2}$ such $(t_1,t_2,\ldots,t_{p+1})$ whose sum is even. 

    If $t_{p+1}$ is odd, then $\sum_{i=1}^p t_i$ should also be odd. There are $\displaystyle\left(d_1\cdots d_p - \left\lceil\frac{\prod_{i=1}^p d_i}{2}\right\rceil\right)\frac{d_{p+1}}{2} $ choices for $(t_1,\ldots,t_{p+1})$ whose sum is even. 

    Thus, the number of $(t_1,t_2,\ldots,t_{p+1})$ such that $\sum_{i=1}^{p+1}t_i$ is even is \begin{eqnarray*}
      &~&  \left\lceil\frac{\prod_{i=1}^p d_i}{2}\right\rceil\frac{d_{p+1}}{2} + \left(d_1\cdots d_p - \left\lceil\frac{\prod_{i=1}^p d_i}{2}\right\rceil\right)\frac{d_{p+1}}{2}\\
      &=& \frac{\prod_{i=1}^{p+1} d_i}{2}\\
      &=& \left\lceil\frac{\prod_{i=1}^{p+1} d_i}{2}\right\rceil
    \end{eqnarray*} 
    Suppose $d_1d_2\cdots d_{p+1}$ is odd, then all the $d_i$'s are odd. Suppose $(t_1,\ldots, t_{p+1})$ is such that $\sum_{i=1}^{p+1}t_i$ is even.

    If $t_{p+1}$ is even, then the number of $(t_1,\ldots, t_{p+1})$ such that their sum is even; is \begin{equation}\label{Ceil1}\displaystyle\Ceil{\frac{\prod_{i=1}^pd_i}{2}}\Ceil{\frac{d_{p+1}}{2}} = \left(\frac{\left(\prod_{i=1}^p d_i\right) + 1}{2}\right)\left(\frac{d_{p+1} + 1}{2}\right)\end{equation}

    When $t_{p+1}$ is odd, then the number of $(t_1,\ldots, t_{p+1})$ such that their sum is even; is \begin{equation}\label{Ceil2}\left(d_1\cdots d_{p+1} - \displaystyle\Ceil{\frac{\prod_{i=1}^pd_i}{2}}\right)\left(d_{p+1} - \Ceil{\frac{d_{p+1}}{2}}\right) = \left(\frac{\left(\prod_{i=1}^p d_i\right) - 1}{2}\right)\left(\frac{d_{p+1} - 1}{2}\right)\end{equation}

    Thus upon adding \ref{Ceil1} and \ref{Ceil2}, we get that the total number of $(t_1,\ldots, t_{p+1})$ such that their sum is even is $$\left(\frac{d_1\cdots d_pd_{p+1 + 1}}{2}\right) = \Ceil{\frac{\prod_{i=1}^{p+1}d_i}{2}}.$$
    Thus, our lemma holds true for $p+1$ as well. So, the lemma is true for all $p$.

\end{proof}
\section{$z$-classes in $C_n$ and $B_n$}\label{$z$-Classes of $C_n$}
The Coxeter groups of $B_n$ and $C_n$ are isomorphic, and they are isomorphic to the wreath product of cyclic group $C_2$ of order $2$ with symmetric group $S_n$ of degree $n.$ We denote this group by $C_2 \wr S_n.$ In order to understand the $z$-conjugcay classes of $C_2 \wr S_n,$ we first study its conjugacy classes.
\subsection{Conjugacy classes in $C_2 \wr S_n$}\label{ConjugacyclassesWCn}
We begin this section by defining wreath product $C_2 \wr S_n.$
\begin{definition}
Let $n$ be a positive integer. Let $(C_2)^n$ denotes the direct product of $n$-copies of cyclic group of order $2.$ The symmetric group $S_n$ acts on $(C_2)^n$ in the following way: For $(a_1,\ldots, a_n)\in (C_2)^n$, and $\sigma \in S_n$:
$$\sigma. (a_1,\ldots, a_n) = (a_{\sigma^{-1}(1)},\ldots, a_{\sigma^{-1}(n)}).$$ Then the semidirect product of $(C_2)^n$ with $S_n$, with respect to the above action, is called {\emph wreath product} of $C_2$ and $S_n$.
\end{definition}

We write an element of $C_2 \wr S_n$ as $[a_1,\ldots,a_n;\sigma]$, where $a_1,\ldots, a_n \in C_2$, and $\sigma \in S_n$. Here, we take $C_2 = \{-1,1\}$. 

Let $\bmf{a}_\sigma = [a_1,\ldots, a_n; \sigma], \bmf{b}_\pi= [b_1, \ldots, b_n; \pi] \in C_2\wr{S_n}$, then the product $$\bmf{a}_\sigma\bmf{b}_\pi=[a_1b_{\sigma^{-1}(1)},\ldots, a_nb_{\sigma^{-1}(n)}; \sigma.\pi].$$
And the inverse of $\bmf{a}$ is $\bmf{a}_\sigma^{-1} = [a^{-1}_{\sigma(1)}, \ldots, a^{-1}_{\sigma(n)}; \sigma^{-1}]$.

Now, we discuss the conjugacy classes of the group $C_2 \wr S_n.$ It is well known fact that conjugacy classes of $S_n$ can be enumerated with partitions of $n.$ We recall the definition of partition of a positive integer. 

\begin{definition}
    A partition of a positive integer $n$ is a sequence of positive integers $\lambda_1 \geq \lambda_2 \geq \dots \geq  \lambda_l >0$ such that $\sum_{i=1}^{l} \lambda_i =n.$ The $\lambda_i$'s are called the parts of the partition. We write this partition of $n$ by $\Lbd = (\lbd_1, \lbd_2,\ldots, \lbd_l).$ 
\end{definition}
Let $\Lbd=(\lbd_1,\lbd_2,…,\lbd_l)$ be a partition of n, then the representative of conjugacy class of $S_n$ enumerated by this partition $\Lbd,$ is $\sigma\in S_n,$ which is product of $l$ cycles, where the length of the $i^{th}$ cycle is $\lbd_i.$ In this case, we say that the cycle type of $\sigma$ is $\Lbd.$

Further, the conjugacy classes of the group $
C_2 \wr S_n$ can be enumerated by signed partitions of $n$ (see the proof of Proposition 24~\cite{Ca}). We will be heavily using this enumeration of conjugacy classes of the group $C_2 \wr S_n$ in the later sections of this article.
Now we define the signed partitions.

\begin{definition}
Let $\Lbd=(\lbd_1,\lbd_2,…,\lbd_l)$ be a partition of $n$. A signed partition $\overline{\Lbd}$ of $n$, of type $\Lbd$, is a partition in which every part $\lbd_i$ can come in $2$ different signs denoted by $\lbd_i$ and $\bar{\lbd_i}.$   
\end{definition}
For example, there are five signed partitions of $2:$
$$11, 1\bar{1},\bar{1}\bar{1},2, \bar{2}  $$
Here $11$, $1\bar{1}$, and $\bar{1}\bar{1}$ are signed partitions of type $11$ whereas $2, \bar{2}$ are signed partitions of type $2.$

To understand the enumeration of conjugacy class of $C_2 \wr S_n$ by signed partitions of $n$, we need to first define the cycle products.

\begin{defn}
Let $\Lbd = (\lbd_1, \lbd_2,\ldots, \lbd_l)$ be a partition of $n.$
Let $\sigma \in S_n$ be a representative of conjugacy class enumerated by partition $\Lbd.$ Let $$\sigma = (i_1,\ldots, i_{\lbd_1})\cdots(k_1,\ldots, k_{\lbd_l}),$$ and let $\bmf{a}_\sigma = [a_1, \ldots, a_n; \sigma]$. Then the {\emph cycle products} of $\bmf{a}_\sigma$ are the products $\prod_{j=1}^{\lbd_1}a_{i_j}$, $\ldots$, $\prod_{j=1}^{\lbd_l}a_{k_i}$. 
\end{defn}
From the above definition, it is clear that there is a cycle product defined for each cycle of $\sigma.$ For example $\Pi_{i=1}^{\lbd_1}a_{i_j}$ is cycle product corresponding to cycle $(i_1,\ldots, i_{\lbd_1})$  of length $\lbd_1.$
For any element $\bmf{a}_\sigma \in C_2 \wr S_n$, the cycle products in $\bmf{a}_\sigma$ are either $-1$ or $1.$ 

The following proposition helps in understanding the conjugacy classes in $C_2 \wr S_n.$

\begin{proposition}[Section 2, \cite{MS}]\label{Conjugacy_classes_C2wrSn}
Let ${\bf a} = [a_1,\ldots, a_n; \sigma]$, and ${\bf b}=[b_1,\ldots, b_n; \pi]$. Then ${\bf a}$ is conjugate in $C_2 \wr S_n$ to $\bmf{b}$ if and only if $\sigma$ is conjugate to $\pi$, and if $\tau \in S_n$ so that $\tau \sigma \tau^{-1} = \pi$, then for each cycle $(i_1, \ldots, i_{\lbd_j})$ in $\sigma$, where $\Lbd = (\lbd_1, \ldots, \lbd_l)$ is the cycle type of $\sigma$, the cycle product $a_{i_1}\cdots a_{i_{\lbd_j}}$ in $\bmf{a}$ is conjugate to the cycle product $b_{\tau(i_1)}\cdots b_{\tau(i_{\lbd_j})}$ in $\bmf{b}$. 
\end{proposition} 
The following proposition gives the correspondence between conjugacy classes of the group $C_2 \wr S_n$ and the signed partitions of $n.$
In the case of $C_2\wr S_n$, the cycle products are either 1 or -1. From the above proposition we see that two elements in $C_2\wr S_n$ are conjugate if and only if they have the same signed cycle type. Thus the conjugacy classes in $C_2 \wr S_n$ are enumerated with the signed partitions $\bar{\Lbd}$ of $n.$ 

\begin{proposition}
Let $\Lbd = (\lbd_1, \ldots, \lbd_l)$ be a partition of $n$. Let $\bmf{a} = [a_1, \ldots, a_n; \sigma]$ be an element of $C_2 \wr S_n$ with $\sigma$ being of cycle type $\Lbd$. Then the signed partition $\overline{\Lbd}$ corresponding to $\bmf{a}$ is written as follows:
 Let $(i_1, \ldots i_{\lbd_j})$ be cycle of length $\lbd_j$ in $\sigma,$ if the cycle product associated with cycle  $(i_1, \ldots i_{\lbd_j})$ is $-1, $ 
 then $\lbd_j$ in $\Lbd$ will be replaced by 
 $\overline{\lbd_j}$ in $\overline{\Lbd},$ otherwise 
 $\lbd_j$ in $\Lbd$ is same as $\lbd_j$ in 
 $\overline{\Lbd}. $ 
 \end{proposition}

\begin{example}
Let  $G=C_2 \wr S_6$. Consider an element $$\bmf{x} = [1,-1,-1,1,1,-1, (145)(26)]\in G$$. The $\sigma$ here is $(145)(26)$, which is of cycle type $\Lbd = 1^12^13^1$. The cycle products are $x_1x_4x_5= 1.1.1 =1$, $x_2x_6 = -1.-1 = 1$, and $x_3 = -1$. Thus the signed permutation for $\bmf{x}$ is $\overline{1}^12^13^1$. 
\end{example}

In the following examples we list the conjugacy classes for $C_2 \wr S_2$ and $C_2 
\wr S_3$, with their representatives, and centralizer sizes.
\begin{example}[$C_2 \wr S_2$]
The conjugacy classes of $C_2 \wr S_2$ are as listed below:
\begin{center}
\begin{tabular}{|c|c|c|}\hline
Class & Representative & Centralizer Size \\ \hline
$1^2$ & $[1,1; (1)]$ & $8$\\ \hline
$\overline{1}1$ & $[-1,1; (1)]$ & $4$\\ \hline
$\overline{1}^2$ & $[-1,-1; (1)]$ & $8$ \\ \hline
$2$ & $[1,1; (12)]$ & $4$ \\ \hline
$\overline{2}$ & $[1,-1;(12)]$ & $4$ \\\hline
\end{tabular}
\end{center}
\end{example}
\begin{example}[$C_2 \wr S_3$]
The conjugacy classes of $C_2 \wr S_3$ are as listed below:
\begin{center}
\begin{tabular}{|c|c|c||c|c|c|}\hline
Class & Representative & $\begin{smallmatrix}\text{Centralizer}\\ \text{Size}
\end{smallmatrix}$ & Class & Representative & $\begin{smallmatrix}\text{Centralizer}\\ \text{Size}
\end{smallmatrix}$ \\ \hline
$1^3$ & $[1,1,1; (1)]$ & $48$ & $2^1\overline{1}^1$ & $[1,1,-1;(12)]$& $8$ \\ \hline
$\overline{1}1^2$ & $[-1,1,1; (1)]$ & $16$ & $\overline{2}^11^1$ &$[1,-1,1;(12)]$ & $8$\\ \hline
$\overline{1}^21^1$ & $[1,-1,-1; (1)]$ & $16$ & $\overline{2}^1\overline{1}^1$ & $[1,-1,-1;(12)]$ & $8$ \\ \hline
$\overline{1}^3$ & $[-1,-1,-1; (1)]$ & $48$ & $3^1$ & $[1,1,1;(123)]$ & $6$\\ \hline
$2^11^1$ & $[1,1,1;(12)]$ & $8$  & $\overline{3}^1$ & $[-1,-1,-1; (123)]$ & $6$\\\hline
\end{tabular}
\end{center}
\end{example}

Now, we define the terminology that we will use throughout this article. 

Let $\Lbd = (\lbd_1, \ldots, \lbd_l)$ be a partition of $n$. If $\lambda_i$ is an even number then it is called even part and if $\lambda_i$ is an odd number then it is called odd part. 

Let $\overline{\Lbd}$ be signed partition of type $\Lbd.$
The part of signed partition $\lambda_i$ without bar is called the part with positive sign, or sometimes \emph{positive part}, and the part $\bar{\lambda_i}$ is called the part with negative sign, or sometimes \emph{negative part}.

Let $\sigma\in S_n$ be of cycle type $\Lbd,$ then we shall sometimes write $Z_{S_n}(\Lbd)$ for  $Z_{S_n}(\sigma).$ Similarly, we shall write  $Z_{C_2 \wr S_n}(\overline{\Lbd})$
for $Z_{C_2 \wr S_n}({\bf a}), $ where $\bf a\in C_2 \wr S_n$ is representative of conjugacy class corresponding to signed partition $\overline{\Lbd}$ of $n.$

\subsection{Centralizers in $C_2 \wr S_n$}
For the sake of convenience, we will write ${\bf h}_k$ for $k$ copies of $h,$ for example the element $[{\bf h_1}_k,{\bf h_2}_k,\ldots, {\bf h_l}_k;\tau] \in C_2 \wr S_n$ contains $k$ copies for $h_i$ for all $1\leq i\leq l.$ We denote the following element of $S_n$ containing $l$ cycles of length $k$ by ${\bf \sigma}_{k.l},$ where $n=kl.$
$${\bf \sigma}_{k.l} = (1,2,\ldots, k)(k+1,k+2, \ldots, 2k)\cdots (n-k+1, n-k+2,\ldots, n).$$ 
\begin{proposition}\label{prop_kl}
Let $n,k \in \mathbb N,$ such that $k|n.$  Let $n = kl$. Then the centralizer of the conjugacy class corresponding to signed partition $k^l$ in $C_2 \wr S_n$ is conjugate to 
$$ \{[{\bf h_1}_k,{\bf h_2}_k,\ldots, {\bf h_l}_k;\tau]\mid h_i = \pm1, \tau\in Z_{S_n}({\bf \sigma}_{k.l}) \},$$ 
Moreover, this centralizer is isomorphic to the group $((C_2\times C_k)\wr S_l).$
\end{proposition}
\begin{proof}
We choose ${\bf a} =[1,1,\ldots, 1 ; {\bf \sigma}_{k.l}]$ as representative of conjugacy class  corresponding to signed partition $k^l$ in $C_2 \wr S_n.$ An element $[h_1, h_2, \ldots, h_n ; \tau]$ commutes with ${\bf a}$ if and only if $\tau \in Z_{S_n}({\bf \sigma}_{k.l})$ and $h_i=h_{{\bf \sigma}_{k.l}}^{-1}(i)$ for all $1\leq i\leq n.$ This implies 
$$Z_{C_2 \wr S_n}({\bf a})= \{[{\bf h_1}_k,{\bf h_2}_k,\ldots, {\bf h_l}_k;\tau]\mid h_i = \pm1, \tau\in Z_{S_n}({\bf \sigma}_{k.l}) \}.$$ It is easy to see now the group $Z_{C_2 \wr S_n}({\bf a})$ is isomorphic to 
$((C_2\times C_k)\wr S_l).$

\end{proof}

\begin{proposition}\label{Prop_overline_kl}
Let $n,k \in \mathbb N,$ such that $k|n.$  Let $n = kl$. 
\begin{enumerate}
    \item If $k$ is odd, then the centralizer of the conjugacy classes corresponding to signed partition $\overline{k}^l$ in $C_2 \wr S_n$ is conjugate to 
$$ \{[{\bf h_1}_k,{\bf h_2}_k,\ldots, {\bf h_l}_k;\tau]\mid h_i = \pm1, \tau\in Z_{S_n}({\bf \sigma}_{k.l}) \},$$ 
Moreover, this centralizer is isomorphic to the group $((C_2\times C_k)\wr S_l).$
  \item If $k$ is even, then the centralizer of the conjugacy classes corresponding to signed partition $\overline{k}^l$ in $C_2 \wr S_n$ is conjugate to 
  $$\left\{ [-{\bf h_1}_{m_1}, {\bf h_1}_{k-m_1},\ldots,-{\bf h_1}_{m_l}, {\bf h_1}_{k-m_l}; \tau]\mid\begin{matrix}  h_i \in C_2, \\ \tau =[\pi_k^{m_1},\ldots,\pi_k^{m_l};\rho] \in Z_{S_n}({\bf \sigma}_{k.l})\end{matrix}
\right\},$$ where $\pi_k=(1,2,\dots,k).$
Moreover, this centralizer is isomorphic to the wreath product $C_{2k}\wr S_l$.
\end{enumerate}

\end{proposition}
\begin{proof}
We begin with the case when $k$ is an odd number.
\begin{enumerate}
    \item Since $k$ is odd, We choose ${\bf a} =[-1,-1,\ldots, -1 ; {\bf \sigma}_{k.l}]$ as representative of conjugacy class  corresponding to signed partition $\overline{k}^l$ in $C_2 \wr S_n.$
    Now the proof in this case is similar to the proof of \ref{prop_kl}.
    \item Now we consider $k$ is an even number. We denote the sequence $\{ -1,\underbrace{1, \ldots, 1}_{k-1~-~~\mathrm{copies}}\}$ by $\overline{\bf 1}_k.$  We choose ${\bf a} = [ \underbrace{\overline{{\bf 1}}_k,\ldots, \overline{{\bf 1}}_k}_{l~-~copies}; {\bf \sigma}_{k.l}]$, as a representative of the conjugacy class corresponding to signed partition $\overline{k}^l.$ If an element ${\bf h}= [h_1, h_2, \ldots, h_n ; \tau]$ commutes with ${\bf a}$, then $\tau \in Z_{S_n}({\bf \sigma}_{k.l}).$ Further $Z_{S_n}({\bf \sigma}_{k.l}) = C_k\wr S_l.$ Now, let $\tau = [f_1,\ldots, f_l;\rho]$, where $f_i \in C_k = \langle (1\cdots k)\rangle$, and $\rho \in S_l$, Now using \cite[Equation 4.1.19]{JK}, we get 
    \begin{equation}\label{equation_action}
          \tau((j-1)k + i) = (\rho(j)-1)k + f_{\rho(j)}(i).
    \end{equation}
  
    With this, we get that, if the element ${\bf h}$ commutes with ${\bf a}$, then we get the following equations, for all $1\leq u \leq l$:
 $$   \begin{smallmatrix}
h_{(u-1)k+1}h_{uk}a_{(\rho^{-1}(u)-1)k+f^{-1}_{u}(1)}&=& a_{(u-1)k+1} &=& -1,\\
h_{(u-1)k+2}h_{(u-1)k+1}a_{(\rho^{-1}(u)-1)k+f^{-1}_{u}(2)}&=& a_{(u-1)k+2}& =& 1\\
\vdots &~& \vdots &~& \vdots \\
h_{uk}h_{uk-1}a_{(\pi^{-1}(u)-1)k+f^{-1}_{u}(k)}&=& a_{uk} &=& 1
\end{smallmatrix}
$$
Let $\pi_k=(1,2,\dots,k)$ and $f_u = \pi_k^{m_u}.$ Then, for all $1\leq u \leq l, $ the above equations implies 
 $$-h_{(u-1)k+1} = \cdots = -h_{(u-1)k+m_u}= h_{(u-1)k+m_u+1} = \cdots = h_{uk}.$$ Hence, we get 
the centralizer $Z_{C_2 \wr S_n}({\bf a})$ is same as  $$\left\{ 
[-{\bf h_1}_{m_1}, {\bf h_1}_{k-m_1},\ldots,{\bf h_l}_{m_l}, {\bf h_l}_{k-m_l}; \tau]\mid  h_i \in C_2, \tau =[\pi_k^{m_1},\ldots,\pi_k^{m_l};\rho]\in Z_{S_n}({\bf \sigma}_{k.l})\right\}.$$ 
The centralizer $Z_{C_2 \wr S_n}({\bf a})$  is isomorphic to the group $C_{2k}\wr S_l$ as given by the following map:
$$[-{\bf h_1}_{m_1}, {\bf h_1}_{k-m_1},\ldots,{\bf h_l}_{m_l}, {\bf h_l}_{k-m_l}; \tau] \mapsto [c^{m_1}, c^{m_2}, \dots, c^{m_l}; \rho ],
$$
where $c$ is generator of cyclic group $C_{2k}.$ 
    \end{enumerate}
\end{proof}
\begin{remark}\label{Remark_split}

Let $n,s,t\in \mathbb N$ such that $n=s+t.$
Let ${\bf s}=[1,2,\dots s]$ and ${\bf t}=[s+1, s+2, \dots n].$ Let $S_{\bf s}$ and $S_{\bf t}$ be groups of permutations of sets $\bf s$ and $\bf t,$ respectively. Clearly $S_{\bf s} \times S_{\bf t}$ is a subgroup of $S_n.$ Further, we can identify $C_2 \wr S_{\bf s} \times C_2 \wr S_{\bf t}$ with a subgroup in $C_2 \wr S_n$ by identifying $([h_1, h_2, \dots, h_s; \rho_s], [h_{s+1}, h_{s+2},\dots h_n; \rho_t])$ with $[h_1, h_2, \dots, h_n; \rho_s\rho_t],$ where $h_i \in C_2$ for all $1\leq i \leq n,$ and $\rho_s \in S_{\bf s}, \rho_t \in S_{\bf t}.$ 

For convenience, we write $S_s$ for the group $S_{\bf s}$ in this article, the reader can understand set ${\bf s}$  from the context. 
\end{remark}

\begin{lemma}\label{Lemma_signed_partition}
Let $n,k \in \mathbb N,$ such that $k|n.$  Let $n = kl$.
Let $s, t\geq 0$ such that $s+t =l$. Then the centralizer $Z_{C_2 \wr S_n}(k^s\overline{k}^t)$ of the conjugacy class corresponding to the signed partition $k^s\overline{k}^t$ is 
conjugate to $Z_{C_2 \wr S_{ks}}(k^s) \times Z_{C_2 \wr S_{kt}}(\overline{k}^t).$
\end{lemma}
\begin{proof}
We consider the representative ${\bf a}=[a_1, a_2, \dots a_n; \sigma_{k.l}]= [{\bf 1}_{sk}, ; \sigma_{k.l}]$ of conjugacy class corresponding to the signed partition $k^s\overline{k}^t$ of $n.$ Using the remark \ref{Remark_split}, we identify the element ${\bf a}$ with element $({\bf a}_1, {\bf a}_2)= ([{\bf 1}_{sk}; \sigma_{k.s}], [\overline{{\bf 1}}_k,\ldots,\overline{{\bf 1}_k}; \sigma_{k.t}]$ of $C_2 \wr S_{ks} \times C_2 \wr S_{kt}, $ where $\sigma_{k.s}$ is product of first $s$ cycles of $\sigma_{k.l},$ and $\sigma_{k.t}$ is product of remaining $t$ cycles. It is clear that $Z_{C_2 \wr S_{ks}}({\bf a}_1) \times Z_{C_2 \wr S_{kt}}({\bf a}_2)$ is contained in $Z_{C_2 \wr S_n}({\bf a}).$

Now let ${\bf h}=[h_1, h_2, \dots h_n ; \tau] \in Z_{C_2 \wr S_n}({\bf a}).$ We get $\tau \in Z_{S_n}(\sigma_{k.l})=C_k \wr S_l,$ where $C_k =\langle (1,2,\dots k)\rangle.$ We denote $\tau =[f_1, f_2, \dots f_l; \rho], $ where $f_i\in C_k$ for all $1\leq i\leq l,$ and $\rho\in S_l.$ The action of $\tau$ is given in the equation  \ref{equation_action}. In addition, we get the following equations:

$$\begin{smallmatrix}
h_{(u-1)k+1}h_{uk}a_{(\rho^{-1}(u)-1)k+f^{-1}_u(1)} = a_{(u-1)k+1} = 1, &~~& h_{(w-1)k+1}h_{wk}a_{(\rho^{-1}(w)-1)k+f^{-1}_w(1)}= a_{(w-1)k+1} = -1,\\
h_{(u-1)k+2}h_{(u-1)k+1}a_{(\rho^{-1}(u)-1)k+f^{-1}_u(2)} = a_{(u-1)k+2} = 1, &~~&h_{(w-1)k+2}h_{(w-1)k+1}a_{(\rho^{-1}(w)-1)k+f^{-1}_w(2)}= a_{(w-1)k+2} = 1\\
\vdots & ~~ & \vdots\\
h_{uk}h_{uk-1}a_{(\rho^{-1}(u)-1)k+f^{-1}_u(k)} = a_{uk} =1,&~~& h_{wk}h_{wk-1}a_{(\rho^{-1}(w)-1)k+f^{-1}_w(k)}= a_{wk} = 1,
\end{smallmatrix} $$
where $1\leq u \leq s$, and $s+1\leq w\leq s+t=l.$
Suppose, $\tau=[f_1, f_2, \dots f_l; \rho]$ is such that for some $1 \leq u \leq s$, $\rho(u) = w$, where $s+1 \leq w\leq l$. Then $a_{(\rho^{-1}(w)-1)k + f_w^{-1}(i)} = a_{(u-1)k+f_w^{-1}(i)} = 1$, for all $1\leq i \leq k$. Thus, we have:
$$\begin{smallmatrix}
h_{(w-1)k+1}h_{wk} = -1\\
h_{(w-1)k+2}h_{(w-1)k+1} = 1\\
\vdots\\
h_{wk}h_{wk-1} = 1
\end{smallmatrix}
$$
hence, we have $h_{wk} = -h_{(w-1)k+1}$, and $h_{wk} = h_{wk-1}= \cdots = h_{(w-1)k+2} = h_{(w-1)k+1},$ which is a contradiction. Thus, the permutation $\tau \in Z_{S_n}(\sigma_{k.l})$ permutes $1,\ldots, sk$ among themselves, and permutes the numbers $sk+1,\ldots,n$ among themselves. Hence $\tau = \tau_s\tau_t$, where $\tau_s \in Z_{S_{sk}}(\sigma_{k.s})$, and $\tau_t \in Z_{S_{kt}}(\sigma_{k.t}),$ ($S_{kt}$ is group of permutations of the set $\{sk+1, \ldots, n\}$).

Therefore $Z_{{S_n}}(\bmf{a})$  $$= \left\{\left[\begin{matrix}\bmf{h_1}_k,\ldots,\bmf{h_s}_k, -\bmf{h_{s+1}}_{m_1} ,\bmf{h_{s+1}}_{k-m_1} ,\ldots,-\bmf{h_{s+t}}_{m_t} ,\bmf{h_{s+t}}_{k-m_t}
\end{matrix}; \tau_s\tau_t\right] \mid \begin{matrix}\tau_s \in Z_{S_{sk}}(\sigma_{k.s})\\ \tau_t \in Z_{S_{tk}}(\sigma_{k.t})\end{matrix}\right\},$$
which is same as  $Z_{C_2 \wr {S_{sk}}}({\bf a}_1)\times Z_{C_2 \wr {S_{tk}}}({\bf a}_2).$ This completes this proof.
\end{proof}
\begin{example}
Consider the group $C_2\wr S_2$ and ${\bf a}= [1,-1 ; (1)] \in C_2 \wr S_n.$ We note that the signed partition corresponding to conjugacy class of ${\bf a}$ is $1\overline{1}.$ We compute:
$$[1,1; (1,2)] [1,-1; (1)] \neq [1,-1 ; (1)][1,1 ; (12)].$$
\end{example}
\begin{theorem}\label{ThOddEven}
Let ${\bf \lambda} =k_1^{l_1}\cdots k_u^{l_u}m_1^{r_1}\cdots m_v^{r_v}$ be a partition of $n,$ where $k_1< \cdots < k_u$ be odd numbers, and $m_1<\cdots <m_v$ be even numbers. Now, for $1\leq i \leq u$, let $s_i, t_i$ be non-negative integers such that $s_i + t_i = l_i$, and like wise for $1\leq j \leq v$, we take $x_j+y_j = r_j$. Now, consider the conjugacy class of $C_2 \wr S_n$ corresponding to the signed partition: 
$${\bf \overline{\lambda}}=k_1^{s_1}\overline{k_1}^{t_1}\cdots k_u^{s_u}\overline{k_u}^{t_u}m_1^{x_1}\overline{m_1}^{y_1}\cdots m_v^{x_v}\overline{m_v}^{y_v}.$$
The centralizer of this conjugacy class is conjugate to:
$$\left(\prod_{i=1}^u (Z_{C_2 \wr {S_{s_ik_i}}}(k_i^{s_i})\times Z_{C_2 \wr{S_{t_ik_i}}}(\overline{k_i}^{t_i})) \right)\times \left(\prod_{j=1}^v (Z_{C_2 \wr{S_{x_jm_j}}}(m_j^{x_j})\times Z_{C_2 \wr {S_{y_jm_j}}}(\overline{m_j}^{y_j})) \right).$$
\end{theorem}
\begin{proof}
    Consider $\sigma\in S_n$  be cycle type ${\bf \lambda}.$ Let $\sigma = \sigma_1\cdots \sigma_u\mu_1\cdots\mu_v$, where for $1\leq i \leq u$, $\sigma_i$ is a product of $l_i$ number of cycles of length $k_i$, and for $1\leq j \leq v$, $\mu_v$ is the product of $r_j$ number of cycles of length $m_j.$ 

From \cite{JK}(Equation 4.1.19), We get 
$$Z_{S_n}(\sigma)= \left(\prod_{i=1}^u Z_{S_{k_il_i}}(\sigma_i)\right)\times \left(\prod_{j=1}^vZ_{S_{m_jr_j}}(\mu_j)\right).$$ Hence any element $\nu \in Z_{S_n}(\sigma)$ is of the form $\tau_1\cdots\tau_u\eta_1\cdots\eta_v$, where $\tau_i\in Z_{S_{k_il_i}}(\sigma_i)$ and $\eta_j \in Z_{S_{m_jr_j}}(\mu_j)$. 

Using the remark \ref{Remark_split}, we get that $Z_{C_2 \wr S_n}({\bf \overline{\lambda}})$ is same as 
$$\left(\prod_{i=1}^u Z_{C_2\wr S_{k_il_i}}(k_i^{l_i})\right)\times \left(\prod_{j=1}^v Z_{C_2 \wr S_{m_jr_j}}(m_j^{r_j})\right).$$ 
Now the result follows from Lemma \ref{Lemma_signed_partition}.

\end{proof}
\subsection{Proof of Theorem \ref{ThMain1}}
Let ${\bf 1}_k $ denotes the sequence $(\underbrace{1,1,\ldots,1}_{k~-~\mathrm{copies}}),$  similarly ${\bf -1}_k$ denotes the sequence$(\underbrace{-1,-1,\ldots,-1}_{k~-~\mathrm{copies}}),$ and 
$\overline{\bmf{1}}_k$ denotes the sequence $(-1,\underbrace{1, \ldots, 1}_{k-1~-~~\mathrm{copies}})$.\
Let $\overline{\Lbd}$ be the signed partition of $n$ as defined in the statement of Theorem \ref{ThOddEven}. We consider the following representative of the conjugacy class corresponding to $\overline{\Lbd}$
$${\bf a}_\sigma = \left( {\bf 1}_{\sum_{i=1}^u s_ik_i}, {\bf -1}_{\sum_{i=1}^u t_ik_i}, {\bf 1}_{\sum_{i=1}^v x_im_i},{\bf \bar{1}}_{\sum_{i=1}^v y_im_i}, \sigma \right),$$

where $\sigma \in S_n$ be of cycle type $\Lbd$ as defined in Theorem \ref{ThOddEven}.
More precisely, if $\sigma_k$ denotes a cycle of length $k$ and $\sigma_k^p$ denotes the product of $p$ number of cycles of length $k,$ then 
 $$ \sigma = \sigma_{k_1}^{s_1+t_1} \sigma_{k_2}^{s_2+t_2}\dots \sigma_{k_u}^{s_u+t_u} \sigma_{m_1}^{x_1+y_1}\sigma_{m_2}^{x_2+y_2}\dots \sigma_{m_v}^{x_v+y_v}.$$  
Moreover, we call $\sigma$ the $S_n$-part of the element ${\bf a}_\sigma.$ 

\begin{lemma}\label{S_n-parts}
    If two elements of $C_2 \wr S_n$ are $z$-conjugate, then their $S_n$-parts are conjugate in $S_n.$
\end{lemma}
\begin{proof}
Let ${\bf a}_\sigma\in C_2 \wr S_n$ as defined above.
Using the Theorem \ref{ThOddEven}, we get

$$ Z_{C_n}(\bmf{a}_\sigma)= \begin{matrix}
                                \displaystyle\prod_{i=1}^u \left(Z_{W(C_{s_ik_i})}([{\bf 1}_{k_i s_i}; \sigma_{k_i}^{s_i}])\times Z_{W(C_{t_ik_i})}([ {\bf -1}_{k_it_i}, \sigma_{k_i}^{t_i}]) \right)\\
                                \displaystyle\prod_{i=1}^v \left( Z_{W(C_{x_im_i})}( [{\bf 1}_{m_ix_i}; \sigma_{m_i}^{x_i}])\times Z_{W(C_{y_im_i})}( [\bmf{\overline{1}}_{m_iy_i};\sigma_{m_i}^{y_i}]).\right)
                               \end{matrix}$$ 
 Now using the fact that center of direct product of groups is same as direct product of centers of groups, we get:
$$Z(Z_{C_n}({\bf a}_\sigma))= \begin{matrix}
 \prod_{i=1}^u\left( \langle [{\bf 1}_{k_is_i}; \sigma_{k_i}^{s_i}],  [ {\bf -1}_{k_is_i}; \sigma_{k_i}^{s_i}] \rangle \times  \langle [ {\bf 1}_{k_it_i}; \sigma_{k_i}^{t_i}], [{\bf -1}_{k_it_i}; \sigma_{k_i}^{t_i})] \rangle\right) \times \\ \prod_{i=1}^v \left(\langle [{\bf 1}_{m_ix_i}; \sigma_{m_i}^{x_i}], [ {\bf -1}_{m_ix_i}; \sigma_{m_i}^{x_i}]\rangle \times \langle [{ \bf \overline{1}}_{m_iy_i};\sigma_{m_i}^{y_i}] \rangle\right). \end{matrix}$$ 
 
 Let $A_{{\bf a}_\sigma}$ be the set of $S_n$-parts of all elements of $Z(Z_{C_n}({\bf a}_\sigma)).$ From the above discussion, we get  $A_{{\bf a}_\sigma} = \langle \sigma_{k_i}^{s_i}, \sigma_{k_i}^{t_i}, \sigma_{m_i}^{x_i}, \sigma_{m_i}^{y_i} \rangle$ is a subgroup of $S_n.$ Let $\eta_{{\bf a}_\sigma}$ be natural action of $S_n$ on set $[1,2,\dots, n]$ restricted on subgroup $A_{{\bf a}_\sigma}.$ The orbits of this action are precisely the cycles of $\sigma.$ Thus, this action determines the cycle type $\Lbd$ of $\sigma.$ 
 
 If ${\bf b}_\pi=[b_1, b_2, \dots, b_n; \pi] \in C_2 \wr S_n$ is $z$-conjugate to ${\bf a}_\sigma$ in $C_n,$ then $A_{{\bf b}_\pi}$ and $A_{{\bf a}_\sigma}$ are conjugate subgroups of $S_n.$ Therefore, orbits of actions $\eta_{{\bf b}_\pi}$ and $\eta_{{\bf a}_\sigma}$ determines the same partition of $n.$ Hence $\sigma$ and $\pi$ are conjugate in $S_n.$

\end{proof}

In the view of above lemma, we get that if two conjugacy classes of $C_2 \wr S_n$ lie in same $z$-class, then the signed partition $\overline{\Lbd}_1$ and $\overline{\Lbd}_2$ corresponding to these two conjugacy classes have the same underlying partition of $n,$ i.e. $\Lbd_1=\Lbd_2.$ Therefore, the following theorem gives all the $z$-classes of $C_2 \wr S_n.$

\begin{theorem}\label{ZThOddEven}
Let $n$ be a positive integer, and let $0 <k_1<k_2 < \cdots < k_u \leq n$ be odd positive integers, and $0 < m_1 < m_2< \cdots < m_v \leq n$ be even positive integers,  and let $l_1, \ldots, l_u$ and $r_1,\ldots, r_v$ be non-negative integers such that $$\sum_{i= 1}^ul_ik_i + \sum_{j=1}^v r_jm_j  =n.$$  
For each $ 1 \leq  i \leq u$: Let $0 \leq t_i \leq s_i \leq l_i$, be such that $s_i + t_i = l_i$. For each $1\leq j \leq v$: Let $0 \leq x_j,v_j \leq r_j$ be such that $x_j + y_j = r_j$. Then the following holds:
The conjugacy classes $$k_1^{c_1}\overline{k_1}^{d_1}\cdots k_u^{c_u}\overline{k_u}^{d_u}m_1^{x_1}\overline{m_1}^{y_1}\cdots m_v^{x_v}\overline{m_v}^{y_v},$$ where for $1\leq i \leq u$, the set $\{c_i ,d_i\} = \{s_i,t_i\}$ are all in the same $z$-class of $C_2 \wr S_n$.
\end{theorem}
\begin{proof} For the signed cycle type $k_1^{s_1}\overline{k_1}^{t_1}\cdots k_u^{s_u}\overline{k_u}^{t_u}m_1^{x_1}\overline{m_1}^{y_1}\cdots m_v^{x_v}\overline{m_v}^{y_v}$, we know from Theorem~\ref{ThOddEven}, that the centralizer of this class is conjugate to $$ \left(\prod_{i=1}^u (Z_{C_2 \wr S_{s_ik_i}}(k_i^{s_i})\times Z_{C_2\wr S_{t_ik_i}}(\overline{k_i}^{t_i})) \right)\times \left(\prod_{j=1}^v (Z_{C_2\wr S_{x_jm_j}}(m_j^{x_j})\times Z_{C_2 \wr S_{y_jm_j}}(\overline{m_j}^{y_j})) \right)$$
Now, consider the conjugacy class corresponding to the following signed partition of $n,$ $$k_1^{c_1}\overline{k_1}^{d_1}\cdots k_u^{c_u}\overline{d_u}^{t_u}m_1^{x_1}\overline{m_1}^{y_1}\cdots m_v^{x_v}\overline{m_v}^{r_v},$$ where $\{c_i, d_i\} = \{s_i, t_i\}$ for each $i = 1,\ldots,u$. 

For the odd parts $k_i$, where $i = 1,\ldots, u$, we get using Proposition \ref{prop_kl} and Proposition \ref{Prop_overline_kl}(1) that the centralizer of $k_i^{c_i}\overline{k_i}^{d_i}$ is conjugate to the centralizer of $k_i^{s_i}\overline{k_i}^{t_i}$. 

From Proposition \ref{Prop_overline_kl}(2), we see that the above does not hold for the even parts. Thus, we conclude that all conjugacy classes $k_1^{c_1}\overline{k_1}^{d_1}\cdots k_u^{c_u}\overline{d_u}^{t_u}m_1^{x_1}\overline{m_1}^{y_1}\cdots m_v^{x_v}\overline{m_v}^{r_v},$ where $\{c_i,d_i\} = \{s_i, t_i\}$ are all in the same $z$-class in $C_2 \wr S_n$. \end{proof}

\begin{proof}[Proof of Theorem~\ref{ThMain1}]
Let $1 \leq k_1 < k_2 < \cdots < k_u \leq n$ be odd positive integers and $2 \leq m_1 < m_2 < \cdots <m_v \leq n$ be even positive integers. Let $\Lbd=k_1^{l_1}\cdots k_u^{l_u}m_1^{r_1}\cdots m_v^{r_v}$ be a partition of $n$. All the possible conjugacy classes of $C_n \wr S_n$ with $\Lbd$ as an underlying partition are $$k_1^{s_1}\overline{k_1}^{l_1-s_1}\cdots k_u^{s_u}\overline{k_u}^{l_u-s_u}m_1^{x_1}\overline{m_1}^{r_1-x_1} \cdots m_v^{x_v}\overline{m_v}^{r_v-x_v},$$ where $0 \leq s_i \leq l_i$ for $1 \leq i \leq u$, and $0 \leq x_j \leq r_j$ for $1 \leq j \leq v$. By Theorem~\ref{ZThOddEven}, we can take $0 \leq s_i \leq \left\lfloor \frac{l_i}{2}\right\rfloor$. So, we have $\left\lfloor \frac{l_i}{2}\right\rfloor + 1$ choices for $s_i$ for each $1 \leq i \leq l_i$; and $r_j + 1$ choices for $x_j$. Hence, we get the Equation~\ref{Emain1}.
\end{proof}

\section{$z$-classes in $D_n$}\label{$z$-Classes of $D_n$}
The Weyl group $D_n$ of type $D_n$ is a subgroup of $C_2 \wr  S_n$ and $[C_2\wr S_n; D_n]=2.$ An element $\bf{a_\sigma} = [a_1,\ldots, a_n; \sigma]$ of $C_2\wr S_n$ lies in $D_n$ if and only if, $\Pi_{i=1}^n a_i =1.$ For more details, One may refer to \cite{Ca, LFV}. 

\subsection{Conjugacy classes in $D_n$ : }\label{ConjugacyClassesW(Dn)}
It is natural to think conjugacy classes in $D_n$ in terms of that of $C_2 \wr S_n$ It is clear that a conjugacy class of $C_2 \wr S_n$ lies in $D_n$ if and only if its associated signed partition has even number of negative parts counting with multiplicity. Like every index 2 subgroup of a group, there are two types of conjugacy classes in $D_n.$ Let $\bf{a_\sigma}$ be a representative of a conjugacy class, corresponding to a signed partition $\bmf{\bar{\lambda}}$, of $C_2 \wr S_n$. Suppose ${\bf a_\sigma}\in D_n.$ 
\begin{enumerate}
    \item Split conjugacy class : A conjugacy class $\bf{a_\sigma}$ of $C_2\wr S_n$ splits into two conjugacy classes in $D_n$ if and only if all parts of its corresponding signed partition are even parts with positive signs, which happens if and only if $Z_{C_2 \wr S_n}({\bf a_\sigma})= Z_{D_n}({\bf a_\sigma}).$
    \item Non-split conjugacy class : The conjugacy class of $\bf{a_\sigma}$ of $C_2 \wr S_n$ remains a single conjugacy class in $D_n$ if and only if its corresponding partition contains at least one odd part or at least one even part with negative sign (or both), which happens if and only if $Z_{D_n}({\bf a_\sigma}) \subsetneq Z_{C_2 \wr S_n}({\bf a_\sigma}).$
\end{enumerate}
For example, out of five conjugacy classes of $C_2 \wr S_2$ corresponding to signed partitions $11, 1\bar{1}, \bar{1}\bar{1}, 2, \bar{2},$ only three conjugacy classes lies in $D_2,$ and these three conjugacy classes correspond to signed partitions $11, \bar{1}\bar{1}, 2.$ Out of these three conjugacy classes, the only one conjugacy class corresponding to signed partition $2$ splits into two conjugacy classes in $D_n.$

If $n$ is an odd number, it is clear from the above discussion that all conjugacy classes of $D_n$ are non-split conjugacy classes. The following subsection gives the $z$-classes of $D_n,$ when $n$ is odd.

\subsection{$z$-classes in $D_n,$ when $n$ is odd :} In this subsection, $n$ is an odd number.
\begin{lemma}\label{LemmaOddWDn}
Let ${\bf x,y}\in D_n.$ The elements ${\bf x,y}$ are $z$-conjugate in $D_n$ if and only they are $z$-conjugate in $C_2 \wr S_n.$
\end{lemma}
\begin{proof}
Since all conjugacy classes of $D_n$ are non-split conjugacy classes. It follows from Lemma \ref{Lemma1_index2}, if the elements ${\bf x,y}$ are $z$-conjugate in $C_2 \wr S_n,$ then they are $z$-conjugate in $D_n.$
    Consider ${\bf x,y}$ are $z$-conjugate in $D_n,$ then 
    $Z_{D_n}({\bf x})= {\bf g}Z_{D_n}({\bf y}){\bf g}^{-1}=Z_{D_n}({\bf gyg^{-1}})$ for some ${\bf g}\in D_n.$ Let ${\bf i}=[-{\bf 1}_n, (1)],$ using ${\bf i}\in Z(C_2 \wr S_n)\setminus D_n,$  we get 
    ${\bf i}\in (Z_{C_2 \wr S_n}({\bf x})\cap Z_{C_2 \wr S_n}({\bf gyg^{-1}}))\setminus D_n.$
Now using lemma \ref{Lemma1_index2_converse}, we get $Z_{C_2 \wr S_n}({\bf x})= Z_{C_2 \wr S_n}({\bf gyg^{-1}})= {\bf g}Z_{C_2 \wr S_n}({\bf y}){\bf g}^{-1},$ and hence ${\bf x}$ and ${\bf y}$ are $z$-conjugate in $C_2 \wr S_n$.

\end{proof}

 \begin{proposition}\label{PropOddWD1}
Each $z$-class in $C_2 \wr S_n$ contains a conjugacy class of $D_n$. 
\end{proposition} 
\begin{proof}
Consider the $z$-class in $C_2 \wr S_n$ corresponding to signed partition $$\bmf{\bar\lambda} =k_1^{s_1}\overline{k_1}^{t_1}\cdots k_u^{s_u}\overline{k_u}^{t_u}m_1^{x_1}\overline{m_1}^{y_1}\cdots m_v^{x_v}\overline{m_v}^{y_v}, $$ where $k_1 < \cdots <k_u$ are odd, and $m_1 < \cdots < m_v$ are even, and for all $1 \leq i \leq u$, we have $s_i \geq t_i$. The total number of negative parts in the signed partition $\bmf{\bar\lambda}$ is $N_0=\sum_{i=1}^u t_i + \sum_{j=1}^v y_j$. Now, if $N_0$ is an even number, then the conjugacy class corresponding to signed partition $\bmf{\bar\lambda}$ is a conjugacy class in $D_n$. 

Suppose $N_0$ is an odd number. Now, as $n$ is odd, there exist $p$ such that $1\leq p \leq u$, such that $s_{p}+t_{p}$ is odd. Now, by Theorem~\ref{ThOddEven}(2), the conjugacy class in $C_2 \wr S_n$ corresponding to the signed partition
$$\bmf{\bar\lambda^\prime} =k_1^{s_1}\overline{k_1}^{t_1}\cdots k_{p}^{p}\overline{k_{p}}^{s_{p}}\cdots k_u^{s_u}\overline{k_u}^{t_u}m_1^{x_1}\overline{m_1}^{y_1}\cdots m_v^{x_v}\overline{m_v}^{y_v},$$
is in the same $z$-class as the conjugacy class  corresponding to signed partition $\bmf{\bar\lambda}$. The number of negative parts of $\bmf{\bar\lambda^\prime}$ is equal to $N_0 + (s_{p}-t_{p})$, is an even number. Therefore, the conjugacy class of $C_2 \wr S_n$ corresponding to the signed partition $\bmf{\bar\lambda^\prime}$ is a conjugacy class of $D_n$.

\end{proof}

\begin{remark}\label{Remark_odd_D_n}
From the Lemma \ref{LemmaOddWDn} and Proposition \ref{PropOddWD1}, we get the number of $z$-classes in $D_n, $ is same as those in $C_2 \wr S_n.$
\end{remark}

\subsection{$z$-classes in $D_n,$ when $n$ is even :} In this subsection, $n$ is an even number. In this case, $D_n$ has both split and non-split conjugacy classes. 
Let $C$ be a conjugacy class in $C_2 \wr S_n,$ and $C$ splits into two conjugacy classes $C_1$ and $C_2$ in $D_n.$ The following theorem gives the condition when the conjugacy classes $C_1$ and $C_2$ are $z$-conjugate in $D_n.$ We recall that split conjugacy classes in $D_n$ correspond to signed partitions with all even parts with positive signs. It follows from the Lemma \ref{Lemma_split}, that if two elements belonging to two distinct split conjugacy classes of $D_n,$ are $z$-conjugate to each other, then they are conjugate in $C_2\wr S_n.$

\begin{theorem}\label{Th_split_CC_in_D_n}
Two split conjugacy classes in $D_n$ corresponding to the signed partition $\bmf{\bar{\lambda}} = m_1^{x_1}\cdots m_v^{x_v}$, where $2\leq m_1 < \cdots m_v\leq n$ are even, are $z$-conjugate if and only if there is a part $m_j = 4q_j + 2$ for some integer $q_j$, whose multiplicity $x_j$ is odd.
\end{theorem}

\begin{proof}
 Let $ \bmf{\bar{\lambda}}=m_1^{x_1}\cdots m_v^{x_v}$ with all even parts. Suppose there is $m_j = 4q_j + 2$, with multiplicity $x_j$ being odd. We shall reorder the parts of $\lbd$ so that $m_1 = 4q_1 + 2$, with odd multiplicity $x_1$. 

Let $\sigma = (1, \dots, m_1)\dots((x_1-1)m_1+1,\dots,x_1m_1)\dots(n-m_v+1,\dots,n)$
 and ${\bf a} = [{\bf 1}^n; \sigma]$ be an element of conjugacy class of $C_2 \wr S_n$ corresponding to signed partition $\bmf{\bar{\lambda}}.$
 Now consider the element ${\bf b} = [\bmf{-1}_{x_1m_1},\bmf{1}_{n-x_1m_1}; \sigma]$. We claim that ${\bf a}$ and ${\bf b}$ are not conjugate in $D_n.$

 Let ${\bf h} = [h_1,\ldots, h_n; \tau] \in C_2\wr S_n$ such that ${\bf hah^{-1}} = {\bf b}$, then $\tau \in Z_{S_n}(\sigma)$ and we have: $h_i h_{\sigma^{-1}(i)} = -1$, for $1 \leq i \leq x_1m_1$ and $h_ih_{\sigma^{-1}(i)} = 1$ for the remaining $i$'s. As $x_1$ and $\frac{m_1}{2}$ are odd, by a routine check, we can see that ${\bf h} \notin D_n$.

 As we are dealing with a split classes, we know that ${Z}_{D_n}({\bf a}) = {Z}_{C_2\wr S_n} ({\bf a})$. 

 Now, we know that 
 \begin{eqnarray*}
{Z}_{D_n}({\bf b}) &=& {Z}_{C_2 \wr S_n}({\bf b}) \\
&=& {Z}_{C_2\wr S_{r_1m_1}([{\bf -1}_{x_1m_1}; \sigma_{m_1.x_1}]}) \times \prod_{i = 2}^v {Z}_{C_2 \wr S_{x_im_i}}
([{\bf 1}_{x_im_i}; \sigma_{m_i.x_i}])\\
&=& {Z}_{C_2\wr S_{r_1m_1}([{\bf 1}_{x_1m_1}; \sigma_{m_1.x_1}]}) \times \prod_{i = 2}^v {Z}_{C_2 \wr S_{x_im_i}}
([{\bf 1}_{x_im_i}; \sigma_{m_i.x_i}])\\
&=& {Z}_{C_2 \wr S_n}({\bf a}) \\
&=& {Z}_{D_n}({\bf a})
\end{eqnarray*}
Thus ${\bf a}$ and ${\bf b}$ are $z$-conjugate in $D_n$.

Now, we consider the signed partition $\bmf{\bar{\lambda}} =m_1^{x_1}\cdots m_v^{x_v}$ corresponding to two split conjugacy classes in $D_n$, where all the parts of the type $m_j = 4q_j + 2$ have even multiplicity $x_j$. Here, we take ${\bf a} = [{\bf 1}_n; \sigma]$ and ${\bf b} = [{\bf -1}_2, {\bf 1}_{n-2}; \sigma]$. One can check that these two elements are not conjugate in $D_n.$

We know that: 
\begin{eqnarray*}
 {Z}_{D_n}({\bf b}) &=& {Z}_{C_2 \wr S_n}({\bf b}) \\
&=&{Z}_{C_2\wr S_{x_1m_1}}([{\bf -1}_2,{\bf 1}_{x_1m_1-2}; \sigma_{m1.x1}]) \times \prod_{i = 2}^v {Z}_{C_2 \wr S_{x_im_i}}([{\bf 1}_{x_im_i}; \sigma_{m_i.x_i}])
\end{eqnarray*}

Suppose ${\bf a}$ and ${\bf b}$ are $z$-conjugate in $D_n.$ We have ${\bf g}{Z}_{D_n}({\bf a}){\bf g}^{-1} = {Z}_{D_n}({\bf b})$ for some ${\bf g} \in D_n$
This implies
\begin{eqnarray*}{\bf g}Z({Z}_{D_n}({\bf a})){\bf g}^{-1}& =& Z({Z}_{D_n}({\bf b}))\\
\Rightarrow{\bf g}\prod_{i=1}^v\langle[{\bf 1}_{x_im_i}; \sigma_{m_i.x_i}],[{\bf -1}^n; (1)]\rangle{\bf g}^{-1} &=& \langle[{\bf -1}_2,{\bf 1}_{x_1m_1-2}; \sigma_{m_1.x_1}],[{\bf -1}_{x_1m_1}, (1)]\rangle\\&~& \times \prod_{i=2}^v\langle[{\bf 1}_{x_im_i}; \sigma_{m_i.x_i}],[{\bf -1}_{n-m_1x_1}; (1)]\rangle
\end{eqnarray*}
Therefore, we get ${\bf g}[{\bf 1}_n; \sigma]{\bf g}^{-1} = [{\bf -1}_2, {\bf 1}_{n-2}; \sigma]^w={\bf b}^w,$ where $(w, o(\bmf{a})) = 1$, in particular, $(w, m_1) = 1$. 

Now, let ${\bf g}=[g_1,\ldots,g_n; \tau]$, where $\tau\sigma\tau^{-1} = \sigma^w,$ 
and we compute 

${\bf b}^w = [-1,{\bf 1}_{w-1},-1,{\bf 1}_{n-w-1}; \sigma^w].$

On equating ${\bf gag}^{-1} = {\bf b}^w$, for $1 \leq i \leq m_1,$ we get
\begin{eqnarray*}
g_1g_{m_1 -w +1} &=& -1\\
g_2g_{m_1 -w + 2} &=& 1\\
\vdots &=& \vdots\\
g_{w+1}g_1 &=& -1\\
\vdots &=&\vdots\\
g_{m_1}g_{m_1 -w} &=& 1
\end{eqnarray*}

We recall that $\sigma_{m_1}^w$ denotes the $(12\cdots m_1)^w.$ As $\sigma_{m_1}^w= (1\sigma_{m_1}^w(1) \cdots \sigma_{m_1}^{w(m_1-1)}(1)),$ we have
$g_{\sigma_{m_1}^{w(m_1-1)}(1)} =g_{\sigma_{m_1}^{w(m_1-2)}(1)} = \cdots = g_{\sigma_{m_1}^w(1)} = -g_1$. 

But  $\{\sigma_{m_1}^w(1), \ldots, \sigma_{m_1}^{w(m_1-1)}(1)\} = \{2,3,\ldots, m_1\}$ we get that $-g_1 = g_2 = \cdots = g_{m_1}.$

For $m_1 + 1 \leq j \leq n$, we have $g_{j}g_{\sigma^{-w}(j)} = 1$. Thus, the element ${\bf g}$ is of the form $${\bf g}= [-d^{(1)}_{1},{\bf d^{(1)}_{1}}_{m_1 - 1},{\bf d^{(1)}_2}_{m_1},\ldots,{\bf d^{(1)}_{x_1}}_{m_1},\cdots, {\bf d^{(v)}_{x_v}}_{m_v}; \tau] ,$$ where $d^{(i)}_j,= \pm 1$, for $1 \leq i \leq v$ and $1 \leq j \leq x_i$. The total number of $-1$ in ${\bf g}$ is either $2q + 1$ or $2q + (m_1 -1)$, which is odd. This implies ${\bf g} \notin D_n$. Hence ${\bf a}$ and ${\bf b}$ are  not $z$-conjugate in $D_n.$ 
\end{proof}

The above theorem implies that if $n=4k+2$ for some $k\in \mathbb N,$ and $C$ is a conjugacy class in $C_2 \wr S_n$ which  splits into two conjugacy classes $C_1$ and $C_2$ in $D_n,$ then $C_1$ and $C_2$ are always $z$-conjugate in $D_n.$ However this does not happen if $n=4k$ for some $k\in \mathbb N.$ For example, the two split conjugacy classes corresponding to signed partition $4$ in $D_4$ are not $z$-conjugate to each other in $D_4.$

Now we shift our attention to the non-split conjugacy classes in $D_n.$ Let $C_1$ and $C_2$ be two non-split conjugacy classes in $D_n.$ Let ${\bf x}\in C_1$ and ${\bf y}\in C_2, $ using Lemma \ref{Lemma1_index2}, we get that if ${\bf x}$ and ${\bf y}$ are $z$-conjugate in $C_2\wr S_n,$ then they are also $z$-conjugate in $D_n.$ The converse is also true in the case of elements belonging to non-split conjugacy classes. The following lemma is the first step in the direction of proving the converse.

\begin{lemma}\label{Center_of_centralizer_in_WDn} Let ${\bf a} \in D_n.$ Then
 $Z(Z_{(D_n}({\bf a}))=Z(Z_{C_2\wr S_n}({\bf a})) \cap D_n$
 unless signed  partition corresponding to the conjugacy class of ${\bf a}$ in $C_2 \wr S_n$ is of type $1^2m_1^{x_1}m_2^{x_2}\dots m_v^{x_v}$ or $\bar{1}^2m_1^{x_1}m_2^{x_2}\dots m_v^{x_v},$ where $m_i$ are even parts with positive sign for all $1\leq i\leq v.$
\end{lemma}
\begin{proof}
 In the view of Subsections \ref{ConjugacyClassesW(Dn)} and Proposition \ref{Regarding_center_of_centralizers}, we that get if the signed partition corresponding to conjugacy class of ${\bf a}$ in $C_2 \wr S_n$ has more than one odd part, or more than one even negative part, or at least one odd part along with at least one even negative part, then we have $Z(Z_{(D_n}({\bf a}))=Z(Z_{C_2\wr S_n}({\bf a})) \cap D_n.$
 
 Hence we are left with the case, when the signed partition corresponding to the conjugacy class of ${\bf a}$ in $C_2 \wr S_n$  is of type $p^qm_1^{x_1}m_2^{x_2}\dots m_v^{x_v},$ where $p$ is either an odd part or an even part with negative sign, with $q$ even, and $m_i$ are even parts with positive signs for all $1\leq i\leq v.$. 
 
 With notations as defined in Remark \ref{Remark_split}, we get
 $$Z_{C_2\wr S_n}({\bf a}) =Z_{C_2\wr S_{pq}}( a_1, \sigma_{p.q}) \times \Pi_{i=1}^v Z_{C_2\wr S_{m_ix_i}}( {\bf 1}_{m_ix_i}, \sigma_{m_i.x_i}),$$
 where $a_i$ is same as ${\bf 1}_{pq}$, or ${\bf -1}_{pq}$, or $\bmf{\bar{1}}_{pq}=(-1,1,\dots,1)$ depending on $p$ is odd part with positive sign, odd part with negative sign, or even part with negative sign. We know that
 $$Z_{C_2\wr S_{m_ix_i}}( {\bf 1}_{m_ix_i}, \sigma_{m_ix_i})= Z_{D_{m_ix_i}}( {\bf 1}_{m_ix_i}, \sigma_{m_i.x_i})$$
 for all $1\leq i\leq v.$ Hence we only need to consider the part $Z_{C_2\wr S_{pq}}( [a_1, \sigma_{p.q}]).$

 We first consider the case when $p=1,$ It is easy to compute
 \[
 Z(D_q)=\begin{cases}
 \{ [{\bf 1}_q, (1)], [{\bf -1}_q, (1)]\} \quad\text{if $q$ is even and $q\neq 2$} \\
 \{[h,h, (1,2)^i]~|~h=\pm1, i=0,1\} \quad\text{if  $q= 2$} \\
 \{ [{\bf 1}_q, (1)] \} \quad\text{if $q$ is odd} \\
 \end{cases}
 \]
 We notice in this case, $Z(D_q)=Z(C_2\wr S_q)\cap D_q$ except when $q=2.$ It gives $Z(Z_{D_q}([{\bf 1}_q;(1)]) )=Z(Z_{C_2\wr S_q}([{\bf 1}_q;(1)])) \cap D_q$ if and only if $q\neq 2.$ 

 Now we consider the case $p=\bar{1},$ in this case $q$ is even as ${\bf a}\in D_n.$ Hence $[{\bf -1}_q;(1)]$ lies in  $Z(C_2\wr S_q)$ as well as in $Z(D_q).$ Therefore again in this case, we get 
 $Z(Z_{D_q}([{\bf -1}_q;(1)]) )=Z(Z_{C_2\wr S_q}([{\bf -1}_q;(1)])) \cap D_q$ if and only if $q\neq 2.$ 
 
When $p \geq 3$, and $q$ is even. Then we know that $Z(Z_{W(C_n)}([a_1,\ldots,a_{pq}; \sigma_{p.q}])) = \langle [a_1,\ldots,a_{pq}; \sigma_{p.q}], [\bmf{-1}_{pq}; (1)] \rangle \subseteq W(D_{pq})$. We claim that $Z(Z_{W(D_{pq})}([a_1,\ldots,a_{pq}; \sigma_{p.q}])) = \langle [\bf{1}_{pq}; \sigma_{p.q}], [\bf{-1}_{pq}; (1)] \rangle$.

When $p \geq 3$ is odd $[a_1,\ldots,a_{pq}; \sigma_{p.q}] = [\bmf{1}_{pq}; \sigma_{p.q}]$. Then $$Z_{W(D_{pq})}([\bmf{1}_{pq}; \sigma_{p.q}]) =\left\{ [\bmf{h_1}_p,\bmf{h_2}_p,\ldots, \bmf{h_q}_p;\tau]\mid h_i = \pm1, \tau\in Z_{S_{pq}}(\sigma_{p.q}), h_1 h_2\cdots h_q = 1\right\}. $$

Suppose we have $\bmf{c} = [-\bmf{1}_p, -\bmf{1}_p, \bmf{1}^{q-2}_p; (\sigma_{p.q})^j] \notin \langle [\bmf{1}_{pq}; \sigma_{p.q}], [-\bmf{1}_{pq}; (1)]\rangle$. Take the element $\bmf{d} = [-\bmf{1}_p, \bmf{1}_p, -\bmf{1}_p, \bmf{1}^{q-3}_p; \tau] \in Z_{W(D_{pq})}([\bmf{1}_{pq}; \sigma_{p.q}])$, such that $\tau (2p +j) = j'$, for $1\leq j ,j'\leq p$. Then \begin{eqnarray*}\bmf{dcd}^{-1} &=& [-\bmf{1}_p, \bmf{1}_p, -\bmf{1}_p, \bmf{1}^{q-3}_p; \tau][-\bmf{1}_p, -\bmf{1}_p, \bmf{1}^{q-2}_p; (\sigma_{p.q})^j][-\bmf{1}_p, \bmf{1}_p, -\bmf{1}_p, \bmf{1}^{q-3}_p; \tau]^{-1}\\ &=& [\bmf{1_p},b_{p+1},\cdots, b_{pq} ;(\sigma_{p.q})^j]\\
&\neq& \bmf{c}.
\end{eqnarray*}
Thus $Z(Z_{W(D_{pq})}([\bmf{1}_{p.q}; \sigma_{p.q}])) = Z(Z_{W(C_n)}([\bmf{1}_{p.q}; \sigma_{p.q}]))$. The same holds for the class $\overline{p}^q$.

For $\overline{p}^q$, where $p$ and $q$ are both even, we use similar arguments and show $Z(Z_{W(D_{pq})}([\overline{1}_{p},\ldots,\overline{1}_p; \sigma_{p.q}])) = Z(Z_{W(C_n)}([\overline{1}_{p},\ldots,\overline{1}_p; \sigma_{p.q}]))$
\end{proof}

\begin{lemma}\label{non-split-in-D_n}
Let $C_1$ and $C_2$ be two non-split conjugacy classes in $D_n.$ Let ${\bf x}\in C_1$ and ${\bf y}\in C_2, $ if ${\bf x}$ and ${\bf y}$ are $z$-conjugate in $D_n,$ then they are also $z$-conjugate in $C_2\wr S_n.$ 
\end{lemma}
\begin{proof}
In the case when the signed  partition corresponding to conjugacy class of ${\bf x}$ in $C_2 \wr S_n$ is not of type $1^2m_1^{x_1}m_2^{x_2}\dots m_v^{x_v}$ or $\bar{1}^2m_1^{x_1}m_2^{x_2}\dots m_v^{x_v},$ where $m_i$ are even parts with positive sign for all $1\leq i\leq v,$ the result follows from Lemma \ref{Center_of_centralizer_in_WDn} and  Lemma \ref{Lemma2_index2_converse}.

Now we are only left with the case when the signed  partition corresponding to conjugacy class of ${\bf x}$ in $C_2 \wr S_n$ is of type $1^2m_1^{x_1}m_2^{x_2}\dots m_v^{x_v}$ and that corresponding to conjugacy class of ${\bf y}$ in $C_2 \wr S_n$ is of type $\bar{1}^2m_1^{x_1}m_2^{x_2}\dots m_v^{x_v},$ where $m_i$ are even parts with positive sign for all $1\leq i\leq v.$ It is easy to see in this case, the elements ${\bf x}$ and ${\bf y}$ are $z$-conjugate in both $C_2 \wr S_n$  and  $D_n.$ 
\end{proof}

The above lemma along with Lemma \ref{Lemma1_index2}, gives us that two elements belonging to two distinct non-split conjugacy classes are $z$-conjugate in $D_n$ if and only if they are $z$-conjugate in $C_2 \wr S_n.$ Now we are only left with the case when two elements are $z$-conjugate in $D_n,$ where one belongs to split conjugacy class and other belongs to non-split conjugacy classes. The following lemma helps in handling such a situation.

\begin{lemma}\label{S_n_parts_D_n}
If two elements of $D_n$ are $z$-conjugate, then either their $S_n$-parts are conjugate in $S_n,$ or the partition type of their $S_n$-parts is $1^2m_1^{x_1}m_2^{x_2}\dots m_v^{x_v}$ or $2^1m_1^{x_1}m_2^{x_2}\dots m_v^{x_v},$ where $m_i$ are even parts of positive type, also $m_i\geq 4$ for all $1\leq i\leq v.$ 
\end{lemma}

\begin{proof} Using the same notations as in Lemma \ref{S_n-parts}, let ${\bf a}_\sigma \in D_n$ and the partition type of $\sigma$ is not $1^2m_1^{x_1}m_2^{x_2}\dots m_v^{x_v},$  where $m_i$ are even parts of positive type. Using Lemma \ref{Center_of_centralizer_in_WDn}, we get $Z(Z_{(D_n}({\bf a}_\sigma))=Z(Z_{(C_2 \wr S_n}({\bf a}_\sigma)) \cap D_n.$ This in turn implies that the set $A_{{\bf a}_\sigma}$ defined in proof of lemma \ref{S_n-parts}, is same as the set of $S_n$-parts of all elements of $Z(Z_{D_n}({\bf a}_\sigma)).$ Hence its natural action on the  set $[1,2, \dots, n]$ determines the cycle type of $\sigma.$ 

If ${\bf b}_\pi$ is $z$-conjugate to ${\bf a}_\sigma$ in $D_n,$ and the partition type $\pi$ is not $1^2m_1^{x_1}m_2^{x_2}\dots m_v^{x_v},$ where $m_i's$  are even parts of positive type, then the sets  $A_{{\bf b}_\pi}$ and $A_{{\bf a}_\sigma}$ defined in proof of lemma \ref{S_n-parts} are conjugate subgroups of $S_n.$ Therefore, orbits of  their actions on the  set $[1,2, \dots, n]$ determine the same partition of $n.$ Hence $\sigma$ and $\pi$ are conjugate in $S_n$ in this case.

Now we consider the case, when the signed partition corresponding to conjugacy class of  ${\bf a}_\sigma$ is either $1^2m_1^{u_1}m_2^{u_2}\dots m_r^{u_r}$ or $\bar{1}^2m_1^{u_1}m_2^{u_2}\dots m_r^{u_r},$ where $m_i$ are even positive parts. In this case
 $$ Z(Z_{D_n}({\bf a}_\sigma)) =\{[h,h, (1,2)^i]~|~h=\pm1, i=0,1\} \times \Pi_{i=1}^v Z(Z_{C_2 \wr S_{m_ix_i}} ([{\bf 1}_{m_ix_i}, \sigma_{m_ix_i}])).$$
 We compute $A_{\bf{a}_\sigma} = \langle (1,2), \sigma_{m_ix_i} \rangle. $ Now the natural action of $A_{\bf{a}_\sigma}$ on the set $[1,2,\dots, n]$ determines the partition $2^1m_1^{x_1}m_2^{x_2}\dots m_v^{x_v}$ of $n.$ 
 
 This implies if ${\bf b}_\pi$ is $z$-conjugate to ${\bf a}_\sigma$ in $D_n,$ then partition corresponding to $\pi$ is either $1^2m_1^{x_1}m_2^{x_2}\dots m_v^{x_v}$ or $2^1m_1^{x_1}m_2^{x_2}\dots m_v^{x_v},$ where $m_i$ are even parts. If the partition type of $\pi$ is $2^1m_1^{x_1}m_2^{x_2}\dots m_v^{x_v},$ where $m_i$ are even parts, then using Lemma \ref{non-split-in-D_n} and Lemma \ref{Center_of_centralizer_in_WDn}, we get the signed partition corresponding to conjugacy class of ${\bf b}_\pi$ is $2^1m_1^{x_1}m_2^{x_2}\dots m_v^{x_v},$ where $m_i$ are even parts with positive type.  Moreover, we claim that partition corresponding to $\pi$ is $2^1m_1^{x_1}m_2^{x_2}\dots m_v^{x_v},$ then $m_i\geq 4$ for all $1\leq i\leq v.$ To establish this claim, let the partition corresponding to $\pi$ be $2^{x_1+1}m_2^{x_2}\dots m_v^{x_v},$ then size of $Z_{D_n}({\bf b}_\pi)$ is $4^{x_1+1}(x_1+1)! (2m_2)^{x_2} x_2! \dots (2m_v)^{x_v} x_v!.$ Since the signed partition corresponding to ${\bf a}_\sigma$ is either $1^22^{x_1}m_2^{x_2}\dots m_v^{x_v}$ or~~ $\bar{1}^22^{x_1}m_2^{x_2}\dots m_v^{x_v},$ then 
 the size of $Z_{D_n}({\bf a}_\sigma)$ is $4^{x_1+1} x_1!(2m_2)^{x_2} x_2! \dots (2m_v)^{x_v} x_v!.$ As ${\bf b}_\pi$ is $z$-conjugate to ${\bf a}_\sigma$ in $D_n,$ size of $Z_{D_n}({\bf b}_\pi)$ and $Z_{D_n}({\bf a}_\sigma)$ should be same. This will happen only if $x_1=0.$ 
 \end{proof}
 \begin{remark}
 Using the above lemma and \cite[Th. 1.1]{BKS:2019}, we get that if two elements are $z$-conjugate in $D_n,$ then their $S_n$ parts are $z$-conjugate in $S_n.$
 \end{remark}

 \begin{proposition}\label{Split_and_non-split_in_D_n}
Let ${\bf a}, {\bf b} \in D_n.$ Let ${\bf a}$ lie in split conjugacy class and ${\bf b}$ lie in non-split conjugacy class in $D_n.$ The elements ${\bf a}$ and ${\bf b}$ are $z$-conjugate in $D_n$ if and only if the signed partition corresponding to conjugacy class of ${\bf a}$ in $C_2 \wr S_n$ is $2^1m_2^{x_2}\cdots m_v^{x_v},$ and the signed partition corresponding to conjugacy class of ${\bf b}$ in $C_2 \wr S_n$ is $\varepsilon m_2^{x_2}\cdots m_v^{x_v},$ where $\varepsilon = { 1}^2$, or ${ -1}^2$ and $m_i$ is even, $m_i\geq 4$ for all $2\leq i \leq v.$ 
 \end{proposition}
\begin{proof}
    Using Theorem \ref{ZThOddEven}, we know that the conjugacy classes corresponding to signed partitions ${ 1}^2m_2^{x_2}\cdots m_v^{x_v}$ and $\bar{1}^2m_2^{x_2}\cdots m_v^{x_v}$ are $z$-conjugate in $C_2 \wr S_n.$ Now using Lemma \ref{Lemma1_index2}, we get these conjugacy classes are also $z$-conjugate in $D_n.$ Also using Theorem \ref{Th_split_CC_in_D_n}, we get that the two split classes of $D_n$ corresponding to signed partition $2^1m_2^{x_2}\cdots m_v^{x_v},$ where $m_i$ is even, $m_i\geq 4$ for all $2\leq i \leq v,$ are $z$-conjugate in $D_n.$ Now let ${\bf a}=[{\bf 1}_n ; \sigma]$ and ${\bf b}=[{\bf 1}_n,\tau],$ where the partition type of $\sigma$ is  ${1}^2m_2^{x_2}\cdots m_v^{x_v}$ and that of $\tau$ is 
    $2^1m_2^{x_2}\cdots m_v^{x_v},$ where  $m_i$ is even, $m_i\geq 4$ for all $2\leq i \leq v.$ It is easy to check that $Z_{D_n}({\bf a})=Z_{D_n}({\bf b}).$ Hence ${\bf a}$ and ${\bf b}$ are $z$-conjugate in $D_n.$

    For the converse, let ${\bf a}$ be in a split conjugacy class and ${\bf b}$, in a non-split conjugacy class of $D_n,$ and ${\bf a} $ is $z$-conjugate to ${\bf b}.$ Therefore, we get ${\bf g}\in D_n$ such that ${\bf g}Z_{D_n}({\bf a}){\bf g}^{-1}=Z_{D_n}({\bf gag}^{-1})=Z_{D_n}({\bf b}).$ We also have $Z_{D_n}({\bf gag}^{-1}) = Z_{C_2 \wr S_n}({\bf gag}^{-1}).$ Now ${\bf gag}^{-1} \in Z(Z_{D_n}({\bf b}))$ but does not commute with any element outside $D_n$, we see that ${\bf gag^{-1}} \notin Z(Z_{C_2 \wr S_n}({\bf b})).$ Thus $Z(Z_{C_2 \wr S_n}({\bf b})) \cap D_n \subsetneq Z(Z_{D_n}({\bf b}))$. Thus using lemma \ref{Center_of_centralizer_in_WDn}, we get the signed partition corresponding to the conjugacy class of ${\bf b}$ to be  $\varepsilon m_2^{x_2}\cdots m_v^{x_v}$, where $\varepsilon = {1}^2$, or $\bar{1}^2$ and $m_i$ is even, for all $2\leq i \leq v.$ Now using lemma \ref{S_n_parts_D_n} and the fact that ${\bf a}$ is in a split conjugacy class, we get $m_i\geq 4$ for all $2\leq i \leq v,$ and the signed partition corresponding to the conjugacy class of ${\bf a}$ is $2^1m_2^{x_2}\cdots m_v^{x_v}$ in $C_2 \wr S_n.$ This completes the proof. 
\end{proof}

\subsection{Proof of Theorem~\ref{DMain2} :}
\begin{proof}
    For a partition ${\bf \lambda} = k_1^{l_1}\cdots k_u^{l_u}m_1^{w_1}\cdots m_v^{w_v} \in \delta(n)$. Now, consider any consider any signed cycle type with ${\bf \lbd}$ as the underlying cycle type:
    $$\tilde{\lbd} = k_1^{s_1}\overline{k_1}^{t_1}\cdots k_u^{s_u}\overline{k_u}^{s_u}m_1^{x_1}\overline{m_1}^{y_1}\cdots m_v^{x_v}\overline{m_v}^{y_v}$$
    Now $s_i + t_i = l_i$ for $1 \leq i \leq u$ and $x_j + y_j = w_j$ for $1 \leq j \leq v$. In particular $s_{i_0} + t_{i_0} = l_{i_0}$. We know that $l_{i_0}$ is odd. So, $s_{i_0} - t_{i_0}$ is even. Suppose $\sum_{i=1}^u t_i + \sum_{j=1}^vy_j$ is even, then this conjugacy class is already in $D_n$. But, if  $\sum_{i=1}^u t_i + \sum_{j=1}^vy_j$ is odd, consider the conjugacy class: $$k_1^{s_1}\overline{k_1}^{t_1}\cdots k_{i_0}^{t_0}\overline{k_{i_0}}^{s_0}\cdots k_u^{s_u}\overline{k_u}^{s_u}m_1^{x_1}\overline{m_1}^{y_1}\cdots m_v^{x_v}\overline{m_v}^{y_v},$$ which is $z$-conjugate to $\tilde{\lbd}$. In this, the negative parts add up to $\displaystyle s_{i_0} + \sum_{i=1, i\neq i_0}^u t_i + \sum_{j=1}^v y_j =(s_{i_0} - t_{i_0})  + \sum_{i=1}^u t_i + \sum_{j=1}^vy_j $, which is the sum of two odd numbers, hence it is even. So, This $z$-class has conjugacy classes in $D_n$. Using Remark ~\ref{Remark_odd_D_n}, the number of such type of $z$-classes of $D_n$ is \begin{equation}\label{DnZclPt1}\sum_{{\bf \lambda} \in \delta(n) }\left( \left(\prod_{i=1}^u \left(\left\lfloor \frac{l_i}{2}\right\rfloor + 1\right)\right)\left(\prod_{j=1}^v (w_j+1)\right)\right) \end{equation}

    Suppose, ${\bf \lbd} =k_1^{l_1}\cdots k_u^{l_u}m_1^{w_1}\cdots m_v^{w_v} $ is a partition in $\delta'(n)$, where no odd part has odd multiplicity, i.e., $l_i$ is even for all $i$. Now, consider a $z$-class of $C_2 \wr S_n$, represented by the conjugacy class $\tilde{\lbd}$ given by $$k_1^{s_1}\overline{k_1}^{t_1}\cdots k_u^{s_u}\overline{k_u}^{t_u}m_1^{x_1}\overline{m_1}^{y_1}\cdots m_v^{x_v}\overline{m_v}^{y_v}, $$
 where $s_i + t_i =  l_i$ and $s_i \geq t_i$; and $x_j + y_j = w_j$. For $\tilde{\lbd}$ to be in $D_n$, we need $\sum_{i=1}^ut_i + \sum_{j=1}^v y_j $  should be an even number. As $s_i \geq t_i$, we know that $0 \leq t_i \leq \frac{l_i}{2} =\left\lfloor \frac{l_i}{2}\right\rfloor$ for $1 \leq i \leq u$ and $0 \leq y_j \leq w_j$  for $1\leq j \leq v$. 
 By Lemma~\ref{LemmaCount}, the number of such $z$-classes is 
 \begin{equation}\label{DnZclPt2} \sum_{{\bf \lambda} \in \delta^\prime(n) }\left\lceil \frac{\left(\prod_{i=1}^u \left(\left\lfloor \frac{l_i}{2}\right\rfloor + 1\right)\right)\left(\prod_{j=1}^v (w_j+1)\right)}{2}\right\rceil
 \end{equation}

 Now, by Proposition~\ref{Split_and_non-split_in_D_n}, the classes $\varepsilon4^{m_2}\cdots n^{m_n}$($\varepsilon = 1^2 ~or~\overline{1}^2$) and both the split classes with signed cycle type $2^14^{m_2}\cdots n^{m_n}$ are $z$ conjugate. Hence $\zeta(n-2)$ such $z$-classes are counted twice, so we subtract $\zeta(n-2)$ from Equation~\ref{DnZclPt2}. 
 
 But, by  Theorem~\ref{Th_split_CC_in_D_n}, we know that the split classes with signed cycle type $2^{m_2}4^{m_4}\cdots n^{m_n}$, where all $4k + 2$ parts have even multiplicity, form separate $z$-classes, and this adds $|\delta'\left(\frac{n}{2}\right)|$ to Equation~\ref{DnZclPt2}. From the above arguments, Equations~\ref{DnZclPt1} and \ref{DnZclPt2}, we get that the number of $z$-classes in $D_n$, when $n$ is even is:
 $$\begin{matrix}\sum_{{\bf \lambda} \in \delta(n) }\left( \left(\prod_{i=1}^u \left(\left\lfloor \frac{l_i}{2}\right\rfloor + 1\right)\right)\left(\prod_{j=1}^v (w_j+1)\right)\right)+\\ \sum_{{\bf \lambda} \in \delta^\prime(n) }\left\lceil \frac{\left(\prod_{i=1}^u \left(\left\lfloor \frac{l_i}{2}\right\rfloor + 1\right)\right)\left(\prod_{j=1}^v (w_j+1)\right)}{2}\right\rceil-\zeta(n-2)+|\delta^\prime\left(\frac{n}{2}\right)|\end{matrix}$$
 
 \end{proof}

\section{$z$-classes of Coxeter groups of exceptional types}\label{z-classes_of_exceptional_types}
In the following table, we give the number of $z$-classes of Coxeter groups of all the exceptional types. We used the computer algebra software GAP to obtain this table. The code for computing $z$-classes of a group $G$ is given in the end of this section. We have used GAP package SLA \cite{SLA} for Coxeter groups of type $F_4, E_6, E_7,$ and $E_8.$

\begin{center}
\begin{table}
\caption{Table of number of $z$-classes in Irreducible Coxeter groups of exceptional types.}\label{Table-exceptional}
\begin{tabular}{|c|c|c|c|}
\hline
S. No. & Group & $\#$ Conjugacy Classes & $\#$ z-classes \\
\hline
1. & $F_4$ & 25 & 16\\
\hline
2. & $E_6$ & 25 & 24\\
\hline
3. & $E_7$ & 60 & 28\\
\hline
4. & $E_8$ & 112 & 65\\
\hline
5. & $H_3$ & 10 & 4 \\
\hline
6. & $H_4$ & 34 & 15\\
\hline
\end{tabular}
\end{table}
\end{center}

The code to compute $z$-classes of a group using GAP is given below:
 
\begin{center}
\begin{framed}
\begin{verbatim}
#This function returns complete list of z-classes of group g.

CentralizerClasses:=function(g)
local cl,cens,result,c,cen,pos,iso,p;
  if IsPermGroup(g) or IsPcGroup(g) then
    iso:=IdentityMapping(g);
    p:=g;
  else
    iso:=IsomorphismPermGroup(g);
    p:=Range(iso);
  fi;
  cl:=ConjugacyClasses(g);
  cens:=[]; # the iso-images of centralizers
  result:=[];
  for c in cl do
    cen:=Image(iso,Centralizer(c));
    pos:=First([1..Length(cens)],x->Size(cens[x])=Size(cen) and
      RepresentativeAction(p,cens[x],cen)<>fail);
    if pos=fail then
      # centralizer is new, make new class
      Add(cens,cen);
      Add(result,[c]);
    else
      # conjugate centralizer, add to result class
      Add(result[pos],c); 
    fi;
  od;
  return result;
end;
  
\end{verbatim}
\end{framed}
\end{center}

\printbibliography
\end{document}